\pdfoutput=1
\documentclass{amsart}
\usepackage[a4paper,outer=3.5cm,inner=3.5cm,total={14cm,21.8cm}]{geometry}
\usepackage[utf8]{inputenc}
\usepackage[english]{babel}
\usepackage{enumerate} \usepackage[bookmarks=false]{hyperref}
\usepackage{graphicx}
\usepackage{amsmath,amsthm,amssymb,amsxtra}
\usepackage{courier}
\usepackage{listings}
\usepackage{mathtools}
\usepackage{tikz-cd}
\usepackage{cleveref}
\usepackage{stmaryrd}

\RequirePackage{mathrsfs}

\newtheorem{thm}{Theorem}[section]
\newtheorem{Theorem}[thm]{Theorem}
\newtheorem*{Theorem*}{Theorem}
\newtheorem{Lemma}[thm]{Lemma}
\newtheorem{Proposition}[thm]{Proposition}

\newtheorem{Setup}[thm]{Setup}
\newtheorem{Question}[thm]{Question}

\theoremstyle{definition}
\newtheorem{Definition}[thm]{Definition}
\newtheorem{Example}[thm]{Example}
\newtheorem{Remark}[thm]{Remark}

\newcommand{\Lie}{\operatorname{Lie}}
\newcommand{\wh}{\widehat}
\newcommand{\im}{\operatorname{im}}

\newcommand{\Hom}{\mathrm{Hom}}
\newcommand{\HOM}{\underline{\mathrm{Hom}}}
\newcommand{\Perf}{\mathrm{Perf}}
\newcommand{\Spa}{\operatorname{Spa}}
\newcommand{\Spd}{\operatorname{Spd}}

\newcommand{\N}{\mathbb{N}}
\newcommand{\G}{\mathbb{G}}
\newcommand{\GL}{\mathrm{GL}}
\renewcommand{\O}{\mathcal{O}}
\newcommand{\Q}{\mathbb{Q}}

\newcommand{\Z}{\mathbb{Z}}
\newcommand{\m}{\mathfrak{m}}
\newcommand{\et}{{\mathrm{\acute{e}t}}}
\newcommand{\proet}{{\mathrm{pro\acute{e}t}}}
\newcommand{\profet}{{\mathrm{prof\acute{e}t}}}
\newcommand{\Higgs}{\mathrm{Higgs}}

\newcommand{\Bun}{\mathrm{Bun}}
\newcommand{\CO}{\mathcal{O}}

\newcommand{\Spec}{\mathrm{Spec}}

\newcommand{\Homc}{\mathrm{Hom}_\mathrm{cts}}
\newcommand{\Hc}{H^1_\text{cts}}
\newcommand{\Hv}{H^1_\mathrm{v}}
\newcommand{\lie}{\mathfrak{g}}
\newcommand{\an}{\mathrm{an}}
\newcommand{\id}{\mathrm{id}}
\newcommand{\cts}{\mathrm{cts}}
\renewcommand{\tt}{\mathrm{tt}}
\newcommand{\dR}{\mathrm{dR}}
\newcommand{\HT}{\mathrm{HT}}
\newcommand{\HTlog}{\mathrm{HTlog}}
\newcommand{\LS}{\mathrm{LS}}
\newcommand*\isomarrow{% arrow with tilde as used for isomorphisms
	\xrightarrow{\raisebox{-0.35em}{\smash{\ensuremath{\sim}}}}
}  
\newcommand{\sm}{{\mathrm{sm}}}
\newcommand{\lsm}{{\mathrm{lsm}}}

\newcommand{\wt}{\widetilde}
\newcommand{\wtOm}{\widetilde{\Omega}}
\newcommand{\HiggsG}{\Higgs_G}

\title[$p$-adic Simpson correspondences for principal bundles in abelian settings]{
	$p$-adic Simpson correspondences for \\principal bundles in abelian settings
}

\begin{document}	
	\author{Ben Heuer, Annette Werner, Mingjia Zhang}

	\begin{abstract}
		We explore generalizations of the $p$-adic Simpson correspondence on smooth proper rigid spaces to principal bundles under rigid group varieties $G$. For commutative $G$, we prove that such a correspondence exists if and only if the Lie group logarithm is surjective. Second, we treat the case of general $G$ on ordinary abelian varieties, in which case we prove a generalisation of Faltings' ``small'' correspondence to general rigid groups. On abeloid varieties, we also prove an analog of the classical Corlette-Simpson correspondence for principal bundles under linear algebraic groups.

	\end{abstract}
	\maketitle
	%\centerline{\small \textbf{2020 MSC:} 14G22, 14G45, 14F20}
	%\setcounter{tocdepth}{1}
	%\tableofcontents
	
	\section{Introduction}
	Let $K$ be a complete algebraically closed non-archimedean field over $\Q_p$. Let $X$ be a connected smooth proper rigid space over $K$ and fix a point $x_0\in X(K)$. Let $\pi_1^\et(X,x_0)$ be the \'etale fundamental group of $X$, i.e.\ the symmetry group of all connected finite \'etale covers of $X$. The starting point of $p$-adic non-abelian Hodge theory is the question how to relate continuous representations
	\[\pi_1^\et(X,x_0)\to \GL_n(K)\]
	to Higgs bundles on $X$.
	The guiding question of this article is to what extent one can generalise this theory from $\GL_n$ to more general groups. We provide evidence that the correct generality is given by rigid group varieties, the $p$-adic analogs of complex Lie groups. The question we set out to investigate is the following:
	
	\begin{Question}\label{q3}
		For which rigid groups $G$ can we find a reasonable condition ``??'' on $G$-Higgs bundles so that there is a ``$p$-adic Corlette-Simpson correspondence with $G$-coefficients''
		\[\Big\{\begin{array}{@{}c@{}l}\text{continuous homomorphisms}\\
			\pi_1^{\et}(X,x_0) \rightarrow G(K) \end{array}\Big\}\isomarrow  \Big\{\begin{array}{@{}c@{}l}\text{$G$-Higgs bundles on $X$}\\\text{satisfying ??}
		\end{array}\Big\}?\]
	\end{Question}    
	For the definition of $G$-Higgs bundles in this setting,  see \Cref{def:Higgs}. 
	Based on known results for $G=\GL_n$, one can divide this problem into two steps, using Scholze's v-site $X_\mathrm{v}$:
	
	%: Let $\mathfrak g$ be the Lie algebra of $G$, regarded as a rigid vector group.
	%\begin{Definition}
	%	A $G$-Higgs bundle on $X$ is a pair $(E,\theta)$ consisting of  a $G$-torsor $E$ on $X_{\et}$ and a section $\theta\in H^0(X,\mathrm{ad}(E)\otimes \wtOm_X^1)$, where $\mathrm{ad}(E):=\mathfrak g\times^GE$ is the adjoint bundle, such that $\theta\wedge \theta=0$ with respect to the natural Lie bracket induced by $\mathfrak g$.
	%\end{Definition}
	%With this definition, the first question we set out to answer is the following:
	\begin{Question}\label{q1}
		For which rigid groups $G$ can we expect to have an equivalence of categories
		\begin{equation}\label{eq:CS-for-G}
			\big\{\text{$G$-torsors on $X_{\mathrm{v}}$}\big\}\isomarrow  \big\{\text{$G$-Higgs bundles on $X$}\big\}?
		\end{equation}
	\end{Question}
	% There is again the relation to continuous representations  $\rho:\pi_1^{\et}(X,x) \rightarrow G(K)$: 
	In a second step, the relation to the fundamental group is furnished by the pro-finite-\'etale universal cover $\wt X\to X$, a $\pi^\et_1(X,x_0)$-torsor in  $X_\mathrm{v}$. Indeed, forming the pushout of $\wt X\to X$ along continuous representations of $\pi^\et_1(X,x_0)$ defines a natural functor
	\[\Big\{\begin{array}{@{}c@{}l}\text{continuous homomorphisms}\\
		\pi_1^{\et}(X,x_0) \rightarrow G(K) \end{array}\Big\}\to {\big\{\text{$G$-torsors on $X_{\mathrm{v}}$}\big\}}.\]
	\begin{Question}\label{q2}
		For which rigid groups $G$ is this functor fully faithful?
	\end{Question}
	As many aspects of $p$-adic non-abelian Hodge theory are still conjectural (even for $G=\GL_n$), understanding specific cases gives important guidance for further investigations.
	In this article, we achieve progress in
	several directions with an abelian flavor: First, we answer all three questions for commutative $G$. Second, when $X$ is an abeloid variety, we answer \Cref{q3} for linear algebraic $G$  and  \Cref{q1} for ``small'' objects and general rigid $G$. We also show that \Cref{q1} has a positive answer for linear algebraic $G$ on any smooth proper $X$.
	
	This gives some evidence for what to expect for general $G$:
	We can deduce from our results conditions on the logarithm for $G$ which are necessary to understand the Betti side of the desired theory, as well as sufficient conditions that answer \Cref{q2} in general. This gives some concrete indications of how to formulate a conjectural $p$-adic Corlette-Simpson correspondence for principal bundles. 
	We now outline our results in more detail.

	\subsection{Previous results}
	Our study is motivated by the following previous results:
	\vspace{-0.3cm}
	\subsubsection{The case $G=\G_a$}
	In the case $G = \G_a$, the answers to these questions are given by  the Hodge--Tate short exact sequence of $p$-adic Hodge theory of Faltings and Scholze \cite[\S3]{Scholze2012Survey}
	\begin{equation}\label{eq:HT-intro}
		0\to H^1_{\et}(X,\G_a)\to \Hom_{\cts}(\pi_1^\et(X,x_0),\G_a(K))\to H^0(X,\Omega^1_X(-1))\to 0.
	\end{equation}
	It is easy to see that $\G_a$-Higgs bundles are classified by $H^1_{\et}(X,\G_a)\times H^0(X,\Omega^1_X(-1))$, so any splitting of this sequence induces an equivalence as in \Cref{q3}: Here for $G=\G_a$, there is no condition on the Higgs bundle side. Second, we have $ \Hom_{\cts}(\pi_1^\et(X,x_0),\G_a(K)) = H^1_\mathrm{v}(X, \G_a)$ in this case by the Primitive Comparison Theorem, answering \Cref{q1}. We therefore think of \Cref{q3} as generalising \eqref{eq:HT-intro} to more general coefficients $G$.  This explains the name ``$p$-adic non-abelian Hodge theory''.
	
	\vspace{-0.3cm}
	\subsubsection{The case $G=\GL_n$}For vector bundles, i.e.\ $G=\GL_n$, an answer to \Cref{q1} has recently been given in \cite{heuer-paSimpson}, where it is shown that  there is a $p$-adic Simpson correspondence
	\begin{equation}\label{eq:S-for-GLn}
		\big\{\text{vector bundles on $X_\mathrm{v}$}\big\}\isomarrow  \big\{\text{Higgs bundles on $X$}\big\}.
	\end{equation}
	Here a Higgs bundle is a pair $(E,\theta)$ of an analytic vector bundle $E$ on $X$ and an $\O_X$-linear homomorphism $\theta:E\to E\otimes \Omega_X^1(-1)$ such that $\theta\wedge\theta=0$. Towards \Cref{q3},
	following Faltings' influential work \cite{Faltings_SimpsonI}, one can embed continuous homomorphisms of $\pi^\et_1(X,x)$ into the category of ``generalised representations'', which in the language of Scholze's diamonds can be interpreted as the vector bundles on $X_\mathrm{v}$. The equivalence \eqref{eq:S-for-GLn} thus leads to a fully faithful functor
	\begin{equation}\label{eq:CS-for-GLn}
		\Big\{\begin{array}{@{}c@{}l}\text{continuous homomorphisms}\\
			\pi_1^{\et}(X,x) \rightarrow \GL_n(K) \end{array}\Big\}\hookrightarrow  \Big\{\begin{array}{@{}c@{}l}\text{Higgs bundles on $X$}
		\end{array}\Big\}.
	\end{equation} Conjecturally, the essential image should admit a description in terms of a $p$-adic replacement for the condition in the complex theory that the Higgs bundle is semi-stable with vanishing Chern classes, thus yielding a ``$p$-adic Corlette-Simpson correspondence''. Known instances include the case $n=1$ where $G=\G_m$ \cite[Theorem~1.1]{heuer-geometric-Simpson-Pic}, and the case when $X$ is an abeloid variety \cite[Theorem~1.1]{HMW-abeloid-Simpson},  %but the
	where the essential image is the category of pro-finite-\'etale Higgs bundles: Here, a $G$-bundle is called pro-finite-\'etale if it is trivialised by the cover $\wt X\to X$, and a $G$-Higgs bundle is pro-finite-\'etale if its underlying $G$-bundle is.
	In general, however, the problem of determining the essential image of \eqref{eq:CS-for-GLn}  currently seems wide open. % in general.

	One reason why it is fruitful to study \eqref{eq:CS-for-GLn} by passing to  v-vector bundles is that  both categories in \eqref{eq:S-for-GLn}  localise, so that it makes sense to analyse their relation for not necessarily proper $X$. This reinterprets $p$-adic non-abelian Hodge theory as the study of \mbox{v-topological} $G$-torsors on rigid spaces.

	\vspace{-0.3cm}
	\subsubsection{The case $G=\GL_n(\O^+)$}An interesting special case is the rigid group $\GL_n(\O^+)\subseteq \GL_n(\O)$. In this case, $G$-torsors on $X_\mathrm{v}$ are equivalent to locally finite free $\O^+$-modules, which are also called ``integral v-vector bundles''. Such objects have recently been studied by Min--Wang  \cite{min2023integral} to investigate integral structures in $p$-adic non-abelian Hodge theory.
	
	\vspace{-0.3cm}
	\subsubsection{General rigid $G$}
	In \cite[Theorem~1.2]{heuer-Moduli}, it is shown that a sheafified version of \eqref{eq:CS-for-G} holds for any rigid group $G$ on any smooth rigid space $X$. Moreover, \cite[Theorem~6.5]{heuer-Moduli} gives a local version of \eqref{eq:CS-for-G}. This produces some evidence that $p$-adic non-abelian Hodge theory for rigid groups is a promising new line of investigation, which was our initial motivation for studying Questions~\ref{q3}, \ref{q1}, and \ref{q2}.

	\subsubsection{Vanishing Higgs field}
	Beyond $G=\GL_n$, \Cref{q3} has previously only been studied in the case of Higgs field $\theta=0$ for reductive groups $G$ when $X$ is a curve: In this case, Hackstein has constructed a functor from certain $G$-bundles to continuous representations $\pi_1^{\et}(X,x) \rightarrow G(K)$ in \cite{hackstein2008principal}, based on the work of Deninger--Werner \cite{DeningerWerner_vb_p-adic_curves}.

	\subsection{The case of commutative $G$}
	It is clear that an equivalence as in \eqref{eq:CS-for-G} cannot exist for \textit{any} rigid group $G$: For example, for the open subgroup $\G_a^+\subseteq \G_a$ given by the open unit disc, it is known that the ``integral Hodge--Tate sequence''
	\[
	0\to H^1_{\et}(X,\G_a^+)\to \Hom_{\cts}(\pi_1^\et(X,x),\G^+_a(K))\to H^0(X,\Omega^1_X(-1))\]
	is only left-exact, and that describing the image on the right is in general difficult. With this in mind, our first main result is  a sufficient and necessary condition for commutative $G$ under which we can answer \Cref{q1} and \Cref{q2}: We call a commutative rigid group $G$ over $K$ \textit{locally $p$-divisible} if it admits an open subgroup $U\subseteq G$ such that $[p]:U\to U$ is surjective (see \Cref{def:loc-pdiv}).
	For example, $G$ is locally $p$-divisible whenever  $[p]:G\to G$ is surjective.
	In particular, any commutative connected algebraic group is locally $p$-divisible.
	%(\Cref{ex:alg-G-loc-p-div}).
	\begin{Theorem}[\Cref{rslt:commutativecase}] \label{rslt:commutative-intro} Let $X$ be a connected  smooth  proper rigid variety over $K$. Let  $G$ be a commutative rigid group over $K$ that is locally $p$-divisible. 
		%(see \Cref{def:loc-pdiv}). 
		
		\begin{enumerate}[(i)]
			\item Choices of an exponential map %$\exp$
			for $K$ and a flat $B_{\dR}^+/\xi^2$-lift of $X$  induce an equivalence
			\begin{align*}
				\big\{\text{$G$-torsors on $X_\mathrm{v}$} \big\}\isomarrow  \big\{
				\text{$G$-Higgs bundles on $X$}\big\}.
			\end{align*}
			\item The above functor restricts to an essentially surjective functor
			\begin{align*}
				\Big\{\begin{array}{@{}c@{}l}\text{continuous homomorphisms  }\\
					\pi_1^{\et}(X,x) \rightarrow G(K) \end{array}\Big\}\to  \Big\{\begin{array}{@{}c@{}l}\text{ pro-finite-\'etale}\\
					\text{$G$-Higgs bundles on $X$}\end{array}\Big\}.
			\end{align*}
			This is an equivalence if moreover $G$ admits an injective morphism into $\GL_n$. 
		\end{enumerate}
	\end{Theorem}
	The assumption on $G$ to be locally $p$-divisible is also necessary.  The crucial point is that the morphism $\log_G:\wh{G}\to \Lie G$ has to be surjective.
	
	To our surprise, the functor in (ii) fails to be fully faithful in general: There is an unexpected relation to the phenomenon of pro-finite-\'etale uniformisation of abeloids, see \Cref{r:rmk-fully-faihful-new}. This shows that already for commutative $G$, \Cref{q2} can be very subtle.
	
	As an important tool in the proof of \Cref{rslt:commutative-intro}, we prove a rigidity lemma for the universal cover $\wt X$ (\Cref{l:Rigidity}), and we show that  any pointed map from a connected smooth proper rigid space to a rigid group $G$ factors through a maximal abeloid subvariety of $G$ (\Cref{t:morph-proper-to-grp}).
	
	The choices in \Cref{rslt:commutative-intro} (1) are the same as in \cite{heuer-paSimpson}. The choice of a flat $B_{\dR}^+/\xi^2$-lift can be thought of as a choice of splitting of \eqref{eq:HT-intro}. There is a canonical such choice if $X$ comes with a model over a local field. The exponential map is a choice of a continuous group homomorphism $\exp:K\to 1+\m_K$ splitting the $p$-adic logarithm, see \Cref{d:exp}. As we show in 
	%the proof,
	\Cref{t:exp-for-loc-p-divisible-grp}, this induces compatible splittings of $\log_G$ for all locally $p$-divisible commutative $G$, generalising Faltings' result that this holds in the linear algebraic case.
	%\begin{Remark}
	
	%	\end{Remark}

\subsection{The Tannakian approach}
The obvious problem in answering \Cref{q3} is that currently, little is known about the condition ``??'' even in the case of $G=\GL_n$. However, one case which is understood is when $X$ is an abeloid variety:
%(constituting the other instance of an ``abelian'' case alluded to in the title)
By \cite[Theorem 6.1]{HMW-abeloid-Simpson}, there is an equivalence of categories between finite-dimensional continuous representations of the Tate module of $X$ and pro-finite-\'etale Higgs bundles on $X$, depending on similar choices as in \Cref{rslt:commutative-intro}.  The first natural strategy to go beyond  $\GL_n$ in this case is to leverage the Tannakian formalism, 
following Simpson's approach in complex non-abelian Hodge theory \cite[\S6]{Simpson-local-system}: In fact, Simpson shows that one can deduce from the Corlette--Simpson correspondence for $\GL_n$ a generalisation for linear algebraic groups $G$ by using the Tannakian formalism.
We show that a Tannakian argument also works in $p$-adic geometry: For any linear algebraic group $G$ over $K$, there is an equivalence of categories
\[
\Big\{
\text{$G$-Higgs bundles on $X_{\et}$}\Big\}\cong  \Big\{\begin{array}{@{}c@{}l}\text{exact tensor functors}\\
	\mathrm{Rep}_K(G)\to \mathrm{Higgs}(X)\end{array}\Big\}.
\]
where $\mathrm{Rep}_K(G)$ is the category of algebraic representations of $G$ and $\Higgs(X)$ is the category of Higgs bundles on $X$. We show that in the $p$-adic case, one needs to impose some further conditions on $X$ to ensure that the notion of ``pro-finite-\'etale'' bundles can be captured by a Tannakian perspective, but these are satisfied for abeloid varieties.

We use this to deduce from \cite[Theorem~1.1]{HMW-abeloid-Simpson} the following result:
\begin{Theorem}[\Cref{t:tannakian}]
	Let $X$ be an abeloid variety and let $G$ be a linear algebraic group over $K$, considered as a rigid group. Then choices of an exponential map
	and a flat $B_\dR^+/\xi^2$-lift of $X$ induce an equivalence of categories
	\begin{align*}
		\Big\{\begin{array}{@{}c@{}l}\text{continuous homomorphisms  }\\
			\pi_1^{\et}(X,0) \rightarrow G(K) \end{array}\Big\}\cong  \Big\{\begin{array}{@{}c@{}l}\text{ pro-finite-\'etale}\\
			\text{$G$-Higgs bundles on $X$}\end{array}\Big\}.
	\end{align*}
\end{Theorem}
The same strategy will also work e.g.\ for curves once the correct condition ?? in \Cref{q3} is worked out in this case. Recently, Xu has made progress in this direction 
\cite{xu2022parallel}.

\subsection{The $p$-adic Simpson correspondence for small bundles}
When translated to the analytic setting, the first step  in Faltings' %construction of the 
$p$-adic 
Simpson correspondence for $\GL_n$ on curves \cite[\S2]{Faltings_SimpsonI} would be to consider the local case that $X$ is an affinoid smooth rigid space with a toric chart. In this situation, he constructs a ``local $p$-adic Simpson correspondence'', which, translated to our v-topological setup, can be interpreted as an equivalence of categories
\begin{equation}\label{eq:LSC-for-GLn}\big\{\text{small v-vector bundles on $X$} \big\}\isomarrow  \big\{
	\text{small Higgs bundles on $X$} \big\}.
\end{equation}
Here, on either side, smallness is a technical condition which roughly means that the objects are $p$-adically close to trivial. This correspondence does not extend beyond ``small'' objects.

In a ``globalization'' step, Faltings then chooses a flat lift of a semi-stable model of $X$ to $A_{\inf}/\xi^2$ and reinterprets the local correspondence in terms of this lift. This allows him to glue the correspondences on a small covering under certain conditions.

This ``small correspondence'' in terms of a lift has since been the subject of extensive studies: At least in the case of semi-stable reduction, it can be regarded as being completely understood following the work of Abbes--Gros and Tsuji \cite{AGT-p-adic-Simpson}\cite{Tsuji-localSimpson}, and more recently Liu--Zhu \cite{LiuZhu_RiemannHilbert} and Wang \cite{Wang-Simpson} in the rigid analytic setting.

The local correspondence in terms of toric charts admits a generalisation to $G$-torsors  for general rigid groups $G$ \cite[Theorem~6.5]{heuer-Moduli}. In order to produce evidence that this can be globalised, we also prove a special case of the ``small correspondence'' for general $G$: 	
\begin{Theorem}
	Let $X$ be an ordinary abeloid variety over $K$. Then any choice of anti-canonical $p$-divisible subgroup $D\subseteq X[p^\infty]$  induces an equivalence of categories
	\[\big\{\text{small $G$-torsors on $X_\mathrm{v}$} \big\}\isomarrow  \big\{
	\text{small $G$-Higgs bundles on $X$} \big\}.\]
\end{Theorem}
For this we use that the datum of $D$ induces a particularly well-behaved kind of global perfectoid Galois cover $X_\infty\to X$ which locally plays the same role as Faltings' toric chart. 

\medskip

This provides evidence that a ``small correspondence'' exists more generally for any rigid space $X$ and any rigid group variety $G$. In general, this seems difficult to construct with current methods:
For example, in \cite{AGT-p-adic-Simpson}\cite{LiuZhu_RiemannHilbert}\cite{Wang-Simpson}, the correspondence is constructed in terms of period sheaves, and  we do not know how to adapt these to  general $G$. That being said, when $G$ has good reduction, the approaches of \cite{MinWang22}\cite{anschutz2023hodgetate} might offer a geometric strategy to construct such a functor, by considering $G$-bundles on Bhatt--Lurie's Hodge--Tate stack.

\subsection*{Notations and conventions}
\begin{itemize}
	\item $p$: a fixed prime number.
	\item $K$: a complete algebraically closed non-archimedean field extension of $\Q_p$. We write $\mathcal{O}_K$ for its ring of integers and $\mathfrak{m}_K$ for the maximal ideal of $\mathcal{O}_K$.
	\item $\Spa(K)$: the adic space $\Spa(K,\O_K)$. Similarly, for an adic space $\Spa(R,R^+)$, we suppress the integral subring $R^+$ from notation if it equals $R^\circ$, the subring of power bounded elements in $R$.  
	\item $\mathbb B$: the closed unit disc over $K$, i.e., $\Spa(K\langle X\rangle,\O_K\langle X\rangle)$. We write $\mathbb D$ for the open unit disc $\mathbb B(|X|<1)$ over $K$.
	\item $\Perf_{K,\tau}$: the category of perfectoid spaces over $\Spd K$, with the $\tau$-topology, where $\tau$ can be \'et, pro\'et or v. Let $\Perf_{K,\tau}^{\mathrm{aff}}$ be the subcategory of affinoid perfectoid spaces.
	\item By a rigid (analytic) space over $K$, we mean an adic space locally of topologically finite type over $\Spa(K)$. By \cite[(1.1.11)]{huber-Etale-Adic}, there is an equivalence of categories $r$ from quasi-separated (resp.\ separated) rigid spaces in the sense of Tate (which we shall call ``classical rigid spaces'') to quasi-separated (resp.\ separated) adic spaces in the sense of Huber. As separatedness will never be an issue in our considerations, this usage of the terminology of ``rigid spaces'' will be harmless.
	\item Let $f:X\to Y$ be a morphism of classical rigid spaces in the sense of Kiehl. If $f$ is proper in the sense of Kiehl \cite[\S9.6.2]{BGR}, then $r(f):r(X)\to r(Y)$ is a proper morphism of adic spaces in the sense of Huber \cite[Definition 1.3.2]{huber-Etale-Adic}, see \cite[Remark 1.3.19.(iii)]{huber-Etale-Adic}. 
	Since we wish to use some results from the classical theory, we shall for simplicity define a ``proper morphism of rigid spaces'' to be one that is locally on the base of the form $r(f)$ for a (Kiehl-)proper morphism of classical rigid spaces. 
	
	This convention is harmless, because one can show that $r(f)$ is proper if and only if $f$ is proper. This direction is more difficult, and as we do not know a complete reference for it in the literature, we sketch the argument (cf  \cite[footnote 3 on p.\ 12]{Scholze-torsion}, \cite[Remarks 1.3.18-19]{huber-Etale-Adic}): Without loss of generality, $f$ admits a formal model $F$. Suppose that $r(f)$ is proper, then the arguments in \cite[Bemerkung 3.11.17, Lemma 3.11.18, Proposition 3.12.5]{Huber-Bewertungsspektrum} do not require the Noetherian assumption and show that $F$ is proper. By \cite[Corollaries~4.4, 4.5]{Temkin_local-properties}, 
	it follows that $f$ is proper. That being said, we will not use the fact that ``$r(f)$ proper implies $f$ proper'' in this article, and only mention it to justify our convention.
	\item Scholze's diamond functor $(\cdot)^\Diamond$ of \cite[Section 10.2]{berkeley} attaches to any analytic adic space over $K$ a v-sheaf on $\Perf_{K,\mathrm v}$, and this assignment is fully faithful on seminormal rigid spaces over $\Spa K$, c.f. \cite[Proposition 10.2.3]{berkeley}. We will therefore freely switch back and forth between smooth rigid spaces and their associated diamonds. In particular, we will often drop the diamond symbol from notation for simplicity, unless confusion might arise.
	\item $X_\mathrm{v}$ ($X$ being an adic space over $K$): the v-site of the diamond $X^\Diamond$, or equivalently the slice category $\Perf_{K,\mathrm{v}|X^\Diamond}$. We denote by $\nu$ the natural map $X_\mathrm{v}\rightarrow X_\text{\'et}$.
	\item All $G$-torsors are equipped with left $G$-actions.
	\item We use $\Homc$ to denote 
	the set of continuous group homomorphisms and correspondingly we use $H^*_\mathrm{cts}$ to denote continuous group cohomology sets.
	\item For a commutative rigid group $G$ over $K$, we use $G[p^n]$ to denote the sub-rigid group of $p^n$-torsion elements, and we denote by $G[p^\infty]$ the v-sheaf $\varinjlim_nG[p^n]$.
\end{itemize}

\subsection*{Acknowledgements}
This project was funded by the Deutsche Forschungsgemeinschaft (DFG, German Research Foundation),  TRR 326 \textit{Geometry and Arithmetic of Uniformized Structures}, project number 444845124.
It originated from discussions the second and the third author had during the workshop {\it Women in Numbers Europe 4} (Utrecht 2022). We thank the organizers and sponsors of this meeting. Special thanks go to Wies{\l}awa Nizio{\l} who coorganized with the second author the research group on "Relative $p$-adic Hodge Theory", and to the participants of this research group. The third named author is funded by Hausdorff Center for Mathematics and is jointly hosted by the Max-Planck Institute for Mathematics. We would like to thank Lucas Gerth and the anonymous referee for helpful comments on an earlier version. Moreover, we thank Johannes Anschütz, Ian Gleason, \mbox{Arthur-C\'esar} \mbox{Le Bras} and Daxin Xu for helpful conversations.

\section{Rigid group theory}
We start with structural results on rigid group varieties that we will need throughout. These are of some independent interest.  We begin by recalling some technical background:
\subsection{Local structure, Lie algebra and logarithm}\label{s:local-struct-G}
\begin{Definition}
	A rigid group $G$ over $K$ is a group object in the category of adic spaces of locally topologically finite type over $\Spa(K)$.
\end{Definition}
We refer to \cite[\S1]{Fargues-groupes-analytiques} and \cite[\S3]{heuer-G-Torsor} for some background on rigid group varieties. One relevant fact is that since $K$ has characteristic $0$, any rigid group variety is automatically smooth \cite[Proposition 1]{Fargues-groupes-analytiques}. For this reason, it is harmless in the following to switch back and forth between the setting of adic spaces over $K$ and that of locally spatial diamonds over $K$.
\begin{Example}
	\begin{enumerate}
		\item For any algebraic group $G_0$ over $K$, the adic analytification $G_0^{an}$ is a rigid group over $K$. In particular, we have the rigid groups $\G_a^{an}$, $\G_m^{an}$ and $\GL_n^{an}$ for any $n$. As we will work exclusively in the analytic category, we will drop the $-^{an}$ from notation throughout when this is clear from context.
		\item For any smooth admissible formal group scheme $\mathfrak G$ over $\mathcal O_K$, the adic generic fibre $G$ is a rigid group. We say that $G$ has good reduction if it is of this form. For example, the adic generic fibre of the affine formal additive group is an open subgroup $\mathbb G_a^+\subseteq \mathbb G_a$ that can be described as the closed unit disc $\mathbb B$ with additive structure.
		\item For any finite dimensional $K$-vector space $V$, we have a rigid group $V\otimes \mathbb G_a$ defined as the rigid analytification of  $\Spec(K[V^\vee])$. This defines a fully faithful functor
		\[-\otimes_K\mathbb G_a:\{\text{finite dimensional $K$-vector spaces}\}\to \{\text{rigid groups over $K$}\}.\]
		We call a rigid group a rigid vector group if it is in the essential image.
		\item A connected proper rigid group is called an abeloid variety. These are the $p$-adic analogs of complex tori. There is a good structure theory of abeloid varieties due to Lütkebohmert \cite{Lutkebohmert_structure_of_bounded}, generalising Raynaud's rigid analytic structure theory of abelian varieties.
	\end{enumerate}
\end{Example}

\begin{Definition}
	For any rigid group $G$ over $K$, we define its Lie algebra to be
	\[\Lie G:=\mathrm{ker}(G(K[X]/X^2)\rightarrow G(K))\]
	This is a finite dimensional $K$-vector space which inherits the structure of a Lie algebra of dimension $\mathrm{dim}_K\Lie G=\mathrm{dim} G$. We will denote the associated rigid vector group by
	\[ \mathfrak g:=\Lie G\otimes_K\mathbb G_a.\]
	This represents the v-sheaf on $\Perf_K$ defined by $Y\mapsto \Lie G\otimes_K\CO(Y)$.
\end{Definition} 

For any $r\in |K|$, let us denote by $\mathbb B(r)$ the closed rigid disc of radius $r$. Then we have the following result about the local structure of rigid groups:
\begin{Proposition}[{\cite[Corollary 3.8.]{heuer-G-Torsor}}]\label{p:local-structure-rigid-grps}
	Let $G$ be any rigid group over $K$. Then $G$ has a neighbourhood basis of the identity $(G_k)_{k\in \mathbb N}$ of open subgroups $G_k\subseteq G$ of good reduction whose underlying inverse system of rigid spaces is isomorphic to $(\mathbb B(|p^{k}|)^d)_{k\in \mathbb N}$ for $d=\dim G$.
\end{Proposition}
We can describe the system $G_k$ more explicitly using the Lie algebra exponential and logarithm. We now summarise the construction,
for a more detailed discussion, see \cite[Section 3.2]{heuer-G-Torsor}. Let $G$ be a rigid group over $K$, and let $G_0$ be any open rigid subgroup of good reduction. Then the Lie algebra of the formal model of $G_0$ induces an integral subgroup $\mathfrak g^+\subseteq \mathfrak g$ of the Lie algebra of $G$ whose underlying rigid space is isomorphic to $\mathbb B(1)^d$. 

For any $k\in \Q_{>0}$, let $\mathfrak g_k:=p^k\mathfrak g^+$, then by \cite[Proposition 3.5]{heuer-G-Torsor} there is $\alpha>0$ such that for any $k> \alpha$ we have the Lie algebra exponential
\[ \exp:\mathfrak g_k\to G\]
of rigid groups. We denote by
\[G_k:= \mathrm{exp}(\mathfrak{g}_k)\]
the image of $\mathfrak g_k$, this is an open subgroup of $G$. Then:
\begin{Lemma}[{\cite[Lemma~4.20]{heuer-G-Torsor}}]\label{l:exp-isom-on-nbhd-basis}
	$(G_k)_{k\in \mathbb{Q}^+_{>\alpha}}$ is a neighbourhood basis of the unit section, consisting of open sub-rigid groups for which $\exp$ has an inverse mapping
	\[\log: G_k\xrightarrow{\sim} \lie_k.\]
\end{Lemma}

Recall that a linear algebraic group is an algebraic group that admits a faithful algebraic representation, i.e.\ a Zariski-closed homomorphism into $\GL_n$ for some $n$. We have the following analog in rigid geometry:
\begin{Definition}
	We call a rigid group $H$ linear analytic if there is an injective homomorphism of rigid groups $H\to \GL_n$ for some $n$.
\end{Definition}
Using Ado's Theorem, one can show that every rigid group is locally linear analytic:
\begin{Proposition}[{\cite[Corollary~3.9]{heuer-G-Torsor}}]\label{p:ado}
	Let $G$ be any rigid group. Then there is a rigid open subgroup $G_0\subseteq G$ of good reduction that is linear analytic. More precisely, we can find a homomorphism of rigid groups $G_0\to \mathrm{GL}_n$ that is a locally closed immersion and
	that identifies the subgroups $(G_k)_{k>\alpha}$ from \Cref{l:exp-isom-on-nbhd-basis} with the preimages of $1+p^{k}M_n(\O^+)$. 
\end{Proposition}
This sometimes allows us to reduce local statements about $G$ to the case of $\GL_n$.
\subsection{The center}

As for algebraic groups, there is a good notion of a central subgroup in rigid group theory.
\begin{Definition}\label{d:center}
	For any connected rigid group $G$ over $K$, the center $Z(G)$ is defined as the kernel of the adjoint morphism (see \cite[\S3.1]{heuer-G-Torsor})
	\[\mathrm{ad}\colon G\to \mathrm{Aut}(\mathfrak g)\]
	By definition, $Z(G)\subseteq G$ is a Zariski-closed normal subgroup of $G$.
\end{Definition}
\begin{Lemma}\label{l:center}
	Let $G$ be a connected rigid group.
	\begin{enumerate}
		\item $Z(G)$ is commutative and $Z(G)(K)$ is the center of $G(K)$.
		\item There exists a linear analytic group $H$ for which there is a left-exact sequence
		\[ 1\to Z(G)\to G\to H.\]
	\end{enumerate}
\end{Lemma}
\begin{proof}
	(2) is immediate from the definition, since $\mathrm{Aut}(\mathfrak g)\cong \GL_n$. For (1), let $g\in G(K)$ and consider the homomorphism \[c_g:G\to G, \quad h\mapsto ghg^{-1}.\] This is the identity if and only if $g$ is in the center of $G(K)$. Since $G$ is connected, $c_g$ is the identity if and only  if it is the identity in a neighbourhood of $1\in G$. By \cite[Theorem~3.4]{heuer-G-Torsor}, it follows that $c_g=\id_G$ if and only if the induced morphism on tangent spaces $\mathrm{ad}(g)=\Lie(c_g):\Lie G \to \Lie G$ is the identity, i.e.\ if $g\in Z(G)(K)$. 
	
	To see that $Z(G)$ is commutative, it suffices to show that the following morphism of rigid spaces is constant:
	\[Z(G)\times Z(G)\to Z(G),\quad g,h\mapsto ghg^{-1}h^{-1}.\]
	This we can check on $K$-points, where it follows from the first part.  
\end{proof}

\begin{Remark}
	The connectedness assumption is necessary: For  example, if $G$ is any \'etale rigid group, not necessarily commutative, then $\mathrm{ad}$ is trivial.
\end{Remark}

\subsection{Rigid groups as v-sheaves of topological groups}
Let $G$ be any rigid group. Then its associated v-sheaf inherits the structure of a v-sheaf of groups on $\Perf_K$, and thus more generally on the ``big'' category of small v-sheaves. In this subsection, we show that one can in fact endow this with the structure of a sheaf of topological groups in a natural way:

Let $G_0\subseteq G$ be an open subgroup of good reduction with an embedding $G_0\hookrightarrow \GL_n$ as in \Cref{p:ado}, inducing a cofinal system of open neighbourhoods $(G_k)_{k>\alpha}$. Without loss of generality, by replacing $G_0$ with $G_k$, we can assume that we already have an embedding
\[\rho:G_0\hookrightarrow \GL_n(\O^+).\]

For any quasi-compact small v-sheaf $T$, we can define a topology on $G(T)$ as follows: for each $m\in \N$, endow $\GL_n((\O^+/p^m)(T))$ with the discrete topology and \[\textstyle\GL_n(\O^+(T))=\varprojlim_{m\in\N}\GL_n((\O^+/p^m)(T))\] with the inverse limit topology. The subspace topology on $G_0(T)$ inherited from $\GL_n(\O^+(T))$ extends to a topological group structure on $G(T)$ in a natural way:	
\begin{Proposition}\label{p:can-grp-struct}
	Let $T$ be a quasi-compact small v-sheaf.	There is a unique structure of a topological group on $G(T)$ such that $G_0(T)$ is an open subgroup. 
	This topological structure makes $G$ into a v-sheaf of topological groups on $\Perf_{K,v}^{\mathrm{aff}}$.
\end{Proposition}
We note that this group structure is independent of the choice of $\rho$: This follows from the fact that $\rho$ is a locally closed immersion. But we will not need this.

\begin{Remark}
	The quasi-compactness assumption is necessary to get the correct topology: For example, for $G=\G_a$ and $T=\G_a$, the natural topology on $\G_a(\G_a)=\O(\G_a)$ is not the linear one induced by $\G_a^+(\G_a)=\G_a^+(K)=\O_K$, which would be too fine.
\end{Remark}
%\begin{proof}

%To define the topological group structure, we use the following:

The proof of \Cref{p:can-grp-struct} relies on the following Lemma. 
\begin{Lemma}\label{l:tech-lemma-top-grp-str}
	The system of open subgroups $(G_k)_{k\in \N}$ of $G$ satisfies the following: For any $k>\alpha$, any quasi-compact v-sheaf $T$ and any $s\in G(T)$, there exists some $j\in \N$ such that
	\[G_j(T)\subseteq s G_k(T)s^{-1}.\]
\end{Lemma}
\begin{proof}
	The statement is true for $T=\Spa(K)$ since then $c_s:G\to G$, $h\mapsto sgs^{-1}$ is a homomorphism of rigid groups. By \cite[Lemma~3.10.3]{heuer-G-Torsor}, we more precisely know that $c_s$ sends $G_j$ into $G_k$ if $\mathrm{ad}(s)$ sends $\mathfrak g_j$ into $\mathfrak g_k$. This is because  $\mathrm{ad}(s)(\mathfrak g_j)\subseteq \mathfrak g_k$ implies that the following diagram commutes:
	\[
	\begin{tikzcd}
		\mathfrak g_j \arrow[d,"\exp"] \arrow[r, "\mathrm{ad}(s)"] & \mathfrak g_k \arrow[d,"\exp"] \\
		G_j \arrow[r, "c_s"]                        & G.                      
	\end{tikzcd}\]
	It follows that also for more general affinoid perfectoid $T$, the inclusion 	$\mathrm{ad}(s)(\mathfrak g_j(T))\subseteq \mathfrak g_k(T)$ implies that $sG_j(T)s^{-1}\subseteq G_k(T)$, since this is a statement that we can check on $K$-points. The same statement for quasi-compact small v-sheaves follows  by covering $T$ by finitely many affinoid perfectoids.        
	The Lemma now follows from the fact that any choice of $\O_K$-basis of $\mathfrak g_0(K)\cong \O_K^d$ induces an isomorphism $\mathrm{Aut}(\mathfrak g)(T)\cong \GL_d(\O(T))$, and any $\O(T)$-linear map $f:\O(T)^d\to \O(T)^d$ is continuous, hence $p^j\O^+(T)^d\subseteq f^{-1}(p^k\O^+(T)^d)$ for some $j$.
\end{proof}

\begin{proof}[Proof of \Cref{p:can-grp-struct}] We endow each $G(T)$ with the topology for which a system of open neighbourhoods of the identity is given by $(G_k(T))_{k\in \N}$. 
	To see that this defines the structure of a topological group, we need to see that the preimage $W$ of $g\cdot G_k(T)$ under $m:G(T)\times G(T)\to G(T)$ is open. For any $(r,s)\in W$, let $j$ be as in \Cref{l:tech-lemma-top-grp-str} for the section $s$, 
	then 
	\[ rG_j(T)\cdot sG_k(T)\subseteq rsG_k(T)G_k(T)\subseteq gG_k(T).\]
	Thus $W$ contains the open neighbourhood $rG_j(T)\times  sG_k(T)$ of $(r,s)\in W$. 
	
	The continuity of the inverse follows similarly from \Cref{l:tech-lemma-top-grp-str} by
	\[ (sG_k(T))^{-1}=G_k(T)s^{-1}\supseteq s^{-1}G_j(T).\]
	Hence $G(T)$ is a topological group. It is then immediate  from the construction that  $G$, regarded as a diamond, is a presheaf of topological groups which is a sheaf on the level of groups. It therefore suffices to see that for any {quasi-compact small v-sheaf} and any v-cover $\{\wt T\rightarrow T\}$ by {quasi-compact small v-sheaves}, the short exact sequence of groups
	\[ \textstyle0\to G(T)\to G(\wt T)\to G(\wt T\times_T\wt T)\]
	is an equalizer in the category of topological groups, i.e.\ that $G(T)$ carries the subspace topology of $ G(\wt T)$. This subspace topology is defined by the open subgroup of $G(T)$ obtained by pullback of $G_0(\wt T)\subseteq G(\wt T)$. But this is precisely $G_0(T)$ by the sheaf property of $G_0$. It thus suffices to verify the statement for $G=G_0$, for which we can reduce to $\GL_n$, where the statement is clear.
\end{proof}
\begin{Lemma}\label{l:top-on-G(T)-test-by-Hausd-sp}
	For any compact Hausdorff space $S$ and any quasi-compact v-sheaf $T$, we have
	\[ \mathrm{Map}(\underline{S}\times T, G)= \mathrm{Map}_{\cts}(S,G(T))\]
	where $G(T)$ is endowed with  the topology from \Cref{p:can-grp-struct}.
\end{Lemma}
\begin{proof}
	Evaluation defines a map
	\[\mathrm{ev}:\mathrm{Map}(\underline{S}\times T,G)\to \mathrm{Map}(S,G(T)).\]
	For any map $\psi:\underline{S}\times T\to {G}$, to see that its image $\phi:S\to G(T)$ is continuous, it suffices to check this for profinite $S$ since any compact Hausdorff space admits a surjection by a profinite set (automatically a quotient map). Note that the preimage $\phi^{-1}(G_k(T))$ is given by those $s\in S$ such that the restriction of $\psi$ to $\{s\}\times T$ factors through $G_k$. Assume $s$ is such a point, then the preimage $\psi^{-1}(G_k)$ is an open sub-v-sheaf $U\subseteq \underline{S}\times T$ containing $\{s\}\times T$. We need to see that there is an open neighbourhood $s\in V\subseteq S$ such that $\underline{V}\times T\subseteq U$.
	
	For this we use that $\{s\}\times T=\bigcap_V \underline{V}\times T$, for $V$ running through open and closed neighbourhoods of $s$ in $S$. Then \[\textstyle(\bigcap_V \underline{V}\times T )\cap (\underline{S}\times T\setminus U)=\emptyset.\] 
	Since $\underline{S}\times T$ is quasi-compact, there is some $V$ as above such that $\underline{V}\times T\subseteq U$. This shows that $\phi^{-1}(G_k(T))$ is open. We deduce the case of $\phi^{-1}(gG_k(T))$ for any $g\in G(T)$ using instead the morphism $\underline{S}\times T\to G$ defined by $\psi\cdot g^{-1}$. Hence $\phi$ is continuous.
	
	We now construct an inverse of  $\mathrm{ev}$ on the set of continuous maps $\phi:S\to G(T)$. Write $G(T)=\sqcup g_iG_0(T)$ for a set of coset representatives $g_i$ of $G(T)/G_0(T)$, then $\phi$ is the disjoint union of maps $S_i\to g_iG_0(T)$ where $S_i:=\phi^{-1}(g_iG_0(T))$ is such that $S=\sqcup S_i$. By translation, we can again reduce to the case of maps $S\to G_0(T)$. Here the result follows from $G_0\cong \mathbb B^{d}$
	by the following Lemma.
\end{proof}
\begin{Lemma}
	The following evaluation map is bijective:
	\[ \textstyle\mathrm{ev}:\mathrm{Map}(\underline{S}\times T,\mathbb B)\to \mathrm{Map}_{\cts}(S,\O^+(T))=\varprojlim_{n}\mathrm{Map}_{\cts}(S,\O^+/p^n(T))\]
\end{Lemma}
\begin{proof}
	It is clear from universal properties that the last map is bijective. For the first, it suffices to prove this for affinoid perfectoid $T$.
	Using first the universal property of $\mathbb B$ in adic spaces, and then the explicit description of fibre products in perfectoid spaces, we have
	\[\mathrm{Map}(\underline{S}\times T,\mathbb B)=\O^+(\underline{S}\times T)\stackrel{a}{=}\O^+(\underline S)\hat{\otimes}_{\O_K}\O^+(T)=\mathrm{Map}_{\cts}(S,\O^+(T)).\]
	To see that this almost isomorphism is an honest isomorphism, recall that a function $f$ in $\O(\underline{S}\times T)$ is in $\O^+(\underline{S}\times T)$ if and only if it is $\leq 1$ at every point of $\underline{S}\times T$. Since we have $|\underline{S}\times T|=|\underline{S}|\times |T|$, this is equivalent to asking that for every $s\in S$, the specialisation $f(s)\in \O(T)$ is in $\O^+(T)$, or in other words that $f$ lies in  $\mathrm{Map}_{\cts}(S,\O^+(T))$.
	
	It remains to see that the mappings are inverse to each other. %This we can do by specialising to points $s\in S$, for which the statement is clear.
	This is because two maps in $\mathrm{Map}(\underline{S}\times T, G)$ 
	% or $\mathrm{Map}_{\cts}(S,G(T))$ 
	agree if and only if they agree at every $s\in S$, or in other words, there is an injection 
	$G(\underline{S}\times T)\rightarrow \prod_{s\in S}G(s\times T).$
	Indeed, let $S'$ be the set $S$ endowed with the discrete topology. Then the morphism of adic spaces $h:\underline{S}'\times T\to \underline{S}\times T$ induces an injection $\O_{\underline{S}\times T}\to h_{\ast}\O_{\underline{S}'\times T}$.
\end{proof}

\subsection{$G$-torsors and $G$-Higgs bundles}
Throughout this section, let $G$ be a rigid group over $\Spa K$ as before. Since $G$ is smooth, it is harmless to identify $G$ with its associated locally spatial diamond. We may therefore also regard $G$ as a (small) sheaf on the v-site of locally spatial diamonds over $K$.

\begin{Definition}
	\begin{enumerate}
		\item   Let $X$ be any locally spatial diamond, and let $\tau=v$ or $\tau=\et$.  A $G$-torsor on $X_{\tau}$ is a sheaf $E$ on $X_{\tau}$ with a
		left action $m:G \times  E \to  E$ of the group $G$ such that locally on $X_\tau$, there is a $G$-equivariant isomorphism 
		\[G\to E\] 
		where $G$ acts via left-translation on itself. The morphisms of $G$-bundles are the
		$G$-equivariant morphisms of sheaves on $X_\tau$. We also call $E$ a $G$-bundle on $X_\tau$.
		\item
		We denote the category of $G$-bundles on $X_\tau$ by $\Bun_{G}(X_\tau)$.
	\end{enumerate}
\end{Definition}

There is also a geometric perspective of $G$-bundles, which is equivalent. The set of $G$-bundles on $X_\tau$ up to isomorphism is naturally isomorphic to $H^1_{\tau}(X,G)$. We refer to \cite[\S3.3, Proposition~3.6]{heuer-G-Torsor} for more background on the notion of $G$-torsors on diamonds.

In order to go from $G$-bundles to Higgs bundles, we moreover need  a good notion of differentials. We therefore specialise to smooth rigid spaces:
\begin{Definition}
	For any smooth rigid space $X$ over $K$, we set $\wtOm:=\Omega^1_{X|K}(-1)$, where $(-1)$ is a Tate twist, considered as a sheaf on $X_{\et}$. The reason to include this Tate twist is that there is then a canonical identification
	$\wtOm=R^1\nu_\ast\CO$ for $\nu:X_\mathrm{v}\to X_{\et}$.
\end{Definition}

\begin{Definition}\label{def:Higgs}
	Let $X$ be a smooth rigid space and $G$ a rigid group over $\Spa K$. A $G$-Higgs bundle on $X$ is a pair $(E,\theta)$ consisting of 
	\begin{itemize}
		\item a
		$G$-bundle $E$ on $X_\text{\'et}$, and
		\item an element $\theta\in H^0(X,\wtOm\otimes_\CO \mathrm{ad}(E))$ such that $\theta\wedge \theta=0$,
	\end{itemize}
	where $\mathrm{ad}(E):=\mathfrak{g}\times^{G}E$ is the adjoint bundle of $E$, with $G$ acting on its Lie algebra $\mathfrak{g}$ via the adjoint representation. The condition $\theta\wedge\theta=0$ can be understood as saying that in terms of any local choice of basis of $\wtOm$, the coefficients of $\theta$ commute with each other with respect to the Lie bracket on $\mathrm{ad}(E)$ induced by $\mathfrak g$.
	
	A morphism of $G$-Higgs bundles $(E_1,\theta_1)\to (E_1,\theta_1)$ is a morphism of the underlying $G$-torsors $\varphi:E_1\to E_2$ such that the induced map $\wtOm\otimes \mathrm{ad}(E_1)\to\wtOm\otimes \mathrm{ad}(E_2)$ sends $\theta_1$ to $\theta_2$.
\end{Definition}
\begin{Remark}
	\begin{enumerate}
		\item 
		Note that since any morphism between $G$-bundles is an isomorphism, it follows that any morphism of $G$-Higgs bundles is an isomorphism.
		\item In the case of $G=\GL_n$, the category of $\GL_n$-Higgs bundles has the same objects as the category of Higgs bundles of rank $n$, but the morphisms are given only by the isomorphisms of Higgs bundles.
		\item For the reader who would like  more motivation for the definition of $G$-Higgs bundles, we refer to the proof of \Cref{p:Tannaka-HB}.
	\end{enumerate}
\end{Remark}

If the rigid group is commutative, the notion of $G$-Higgs bundles simplifies substantially, because the underlying torsor and the Higgs field can be disentangled:
\begin{Lemma}\label{l:Higgs-for-comm-G}
	Let $G$ be a commutative rigid group. Let $X$ be any smooth rigid space. Then:
	
	\begin{enumerate}
		\item A $G$-Higgs bundle on $X$ is a pair $(E,\theta)$ consisting of a $G$-torsor $E$ on $X_\et$ and a section $\theta\in H^0(X,\mathfrak g\otimes \wtOm)$.
		\item  
		The set of morphisms between two $G$-Higgs bundles  $(E_1,\theta_1)$ and $(E_2,\theta_2)$ is empty unless $\theta_1= \theta_2$, when it is the set of morphisms between the $G$-torsors $E_1$ and $E_2$.
		\item For any $G$-Higgs bundle $(E,\theta)$, we have $\underline{\mathrm{Aut}}(E,\theta)=G$. 
	\end{enumerate}
\end{Lemma}
\begin{proof}
	\begin{enumerate}
		\item Since $G$ is commutative, the adjoint representation $G\to \GL(\mathfrak g)$ is trivial, hence
		\[ \mathfrak g\times^GE=G\setminus(\mathfrak g\times E )= \mathfrak g\times (G\setminus E)=\mathfrak g.\]
		%where the action by $G$ on the second term is via the antidiagonal action.
		\item This follows from part (1) since the induced morphism $\mathfrak g\times^GE_1\to \mathfrak g\times^GE_2$ is identified with the identity on $\mathfrak g$.
		\item This follows from (2) which implies that $\underline{\mathrm{Aut}}(E,\theta)=\underline{\mathrm{Aut}}(E)=G$. Here the last identification uses again that $G$ is commutative.\qedhere
	\end{enumerate}
\end{proof}

For the rest of this subsection, let us assume that $G$ is commutative. Then for any locally spatial diamond $X$ and $\tau=\et$ or v, the set $H^1_\tau(X,G)$ of $G$-torsors on $X_{\tau}$ up to isomorphism has a natural group structure. This comes from the following functorial construction:
\begin{Definition}\label{l:otime^G}
	Let $E_1$ and $E_2$ be $G$-bundles on $X_\tau$. Following  \cite[Proposition~7]{Serre-fibres}, we consider the product $E_1\times E_2$. This carries the natural structure of a $G\times G$-torsor. We let
	\[E_1\otimes^GE_2 := G\times^{G\times G}(E_1\times E_2)\]
	be the pushout of $E_1\times E_2$ along the multiplication map $m:G\times G\to G$.
	This defines a natural symmetric tensor bifunctor, 
	\[-\otimes^G-:\Bun_G(X_\tau)\times \Bun_G(X_\tau)\to \Bun_G(X_\tau),\]
	natural in $G$ and in $X$.  For any $G$-torsor $E$ let $E^{-1}$ be the pushout of $E$ along the inverse map $-^{-1}:G\to G$, then there is a natural isomorphism $E\otimes^GE^{-1}=G$.
\end{Definition}

\subsection{The pro-finite-\'etale universal cover}\label{s:profet-univ-cover}
The reason why we have set up the formalism of $G$-torsors not only for smooth rigid spaces but for more general locally spatial diamonds is that we naturally encounter examples of the latter when dealing with \Cref{q3}. Namely, we will need the $p$-adic analogue of the complex universal cover which we now recall from \cite[Definition~4.6]{heuer-Line-Bundle}:
\begin{Definition}\label{d:univ-cover}
	\begin{enumerate}
		\item 
		Let $X$ be a connected smooth proper rigid space over $K$ and fix a base point $x_0\in X(K)$. The pro-finite-\'etale universal cover of $X$ is the diamond
		\[ \wt X:=\varprojlim_{X'\to X}X'\]
		where the projective limit is indexed over the category of connected finite \'etale covers $X'\to X$ together with a lift $x'\in X'(K)$ of $x_0$. Then $\wt X\to X$ is a pro-\'etale torsor under $\pi^{\et}_1(X,x_0)$, the \'etale fundamental group of $X$. By definition, the base point $x_0$ has a distinguished lift to $\wt X(K)$ that we also denote by $x_0$.
		\item We call a $G$-bundle $E$ on $X_\et$ or $X_\mathrm{v}$ pro-finite-\'etale if the pullback of $E$ along $\wt X\to X$ is a trivial $G$-bundle. We call a $G$-Higgs bundle $(E,\theta)$ pro-finite-\'etale if the underlying $G$-bundle $E$ is pro-finite-\'etale.
	\end{enumerate}
\end{Definition}
\begin{Example}\label{ex:abeloid-univ-cover}
	If $X=A$ is an abeloid variety, then $\wt{A}=\varprojlim_{[N],N\in \N}A$ is represented by a perfectoid space  \cite[Theorem~1]{AWS-perfectoid}. The statement about torsors means in this case that in the category of diamonds, we have $A = \wt{A} / TA$, where $TA = \pi^\et_1(A,0)$ is the adelic Tate module.
\end{Example}
\begin{Example}
	By \cite[Theorem~1.1]{heuer-geometric-Simpson-Pic}, an analytic line bundle on an abeloid variety $A$ is pro-finite-\'etale if and only if it is translation invariant and the associated point $x\in A^\vee(K)$ in the dual abeloid variety $A^\vee$ is topologically torsion. This means that there is $N\in \N$ such that $x^{Np^n}\to 1$ for $n\to \infty$ for the natural non-archimedean topology on $A^\vee(K)$.
\end{Example}

\section{The Tannakian approach for linear algebraic groups}

In complex geometry, the non-abelian Hodge correspondence of Corlette and Simpson can be extended from $\GL_n$ to any reductive group. In fact,  for any linear algebraic group $G$, the equivalence of categories between $G$-representations of the topological fundamental group and semi-stable $G$-Higgs bundles with vanishing Chern classes  can be deduced from the statement for $\GL_n$ by the Tannakian formalism \cite[\S6]{Simpson-local-system}.

The goal of this section is to show that the Tannakian formalism can also be leveraged to deduce various instances of the $p$-adic non-abelian Hodge correspondence for linear algebraic $G$ from the (in part still conjectural) case of $\GL_n$. In particular, we obtain a $p$-adic Corlette-Simpson correspondence for $G$ whenever the correspondence for $\GL_n$ is known. We note that this approach only stands a chance to work for linear algebraic groups, and does not generalise to more general algebraic or rigid analytic groups (for example, any homomorphism from an abelian variety to $\GL_n$ is clearly trivial).

Our main result in this section is the following partial answer to \Cref{q3} for abeloid varieties, which we will deduce using the Tannakian formalism from the case of $\mathrm{GL}_n$ treated in \cite{HMW-abeloid-Simpson}:
\begin{Theorem}[Corlette-Simpson correspondence for $G$-bundles]\label{t:tannakian}
	Let $X$ be an abeloid variety. Let $G$ be the analytification of a linear algebraic group $G^\mathrm{alg}$ over $K$. Then choices of a flat $B_{\dR}^+/\xi^2$-lift $\mathbb X$ of $X$ and of an exponential (see \Cref{d:exp})  induce an equivalence of categories
	\begin{align*}
		\Big\{\begin{array}{@{}c@{}l}\text{continuous homomorphisms  }\\
			\pi_1^{\et}(X,0) \rightarrow G(K) \end{array}\Big\}\cong  \Big\{\begin{array}{@{}c@{}l}\text{ pro-finite-\'etale}\\
			\text{$G$-Higgs bundles on $X$}\end{array}\Big\},
	\end{align*}
	where $G(K)$ is equipped with the topology induced from that on $K$.
\end{Theorem}
For the proof, we begin by setting up the Tannakian formalism in this context.

\begin{Definition}\label{d:alg-rep}
	\begin{enumerate}
		\item For any rigid group $G$ which is the analytification of a linear algebraic group $G^\text{alg}$, we denote by $\mathrm{Rep}_K(G)$ the tensor category of finite dimensional algebraic representations of $G^\text{alg}$.
		\item For any topological group $\pi$, let  $\mathrm{Rep}_K(\pi)$ be the tensor category of continuous representations of $\pi$
		on finite dimensional $K$-vector spaces. 
	\end{enumerate}
\end{Definition}    
\begin{Proposition}\label{p:tannakian-g-torsors-on-sousperf}
	Let $X$ be any sousperfectoid space in the sense of \cite[\S 6.3.1]{berkeley}, or any rigid space. Then for any $G$ as in \Cref{d:alg-rep}.(1), there is an equivalence of categories
	\[
	\Big\{
	\text{$G$-bundles on $X_{\et}$}\Big\}\isomarrow  \Big\{\begin{array}{@{}c@{}l}\text{exact tensor functors}\\
		\mathrm{Rep}_K(G)\to \mathrm{Bun}(X_\et)\end{array}\Big\},
	\]
	natural in $X$, where $\mathrm{Bun}(X_\et)$ is the tensor category of vector bundles on $X_\et$. 
\end{Proposition}
\begin{proof}
	In the sousperfectoid case, this is \cite[Theorem~19.5.2]{berkeley}. In fact, the proof in  \textit{loc.\ cit.\ } only uses the sousperfectoid assumption to guarantee sheafiness of the structure sheaf on the \'etale site, hence it still works in the Noetherian setting of rigid spaces.
\end{proof}

\begin{Proposition}\label{p:pirep}
	Let $\pi$ be any topological group.  Then there is an equivalence of categories
	\begin{align*}
		\Big\{\begin{array}{@{}c@{}l}\text{continuous homomorphisms  }\\
			\pi \rightarrow G(K) \end{array}\Big\}\cong  \Big\{\begin{array}{@{}c@{}l}\text{exact tensor functors}\\
			\mathrm{Rep}_K(G)\to \mathrm{Rep}_K(\pi)\end{array}\Big\}.
	\end{align*}
	Here on the left hand side, the set of morphisms $\rho_1\to \rho_2$ between any two continuous representations is given by the set of $g\in G(K)$ such that $g\rho_1g^{-1}=\rho_2$.
\end{Proposition}
\begin{proof}
	Given a continuous homomorphism $f:\pi\to G(K)$, we obtain an exact tensor functor  on the right hand side by sending any representation $(V,\rho: G^\text{alg}\rightarrow \mathrm{GL}(V))$ to 
	\[(V, \pi\xrightarrow{f} G(K)\xrightarrow{\rho^\text{an}(K)}\mathrm{GL}(V)(K))\]
	where $G(K)$ and $\mathrm{GL}(V)(K)$ are endowed with their topologies from \Cref{p:can-grp-struct}.
	
	Conversely, denote by $K$-Vect the category of finite dimensional topological $K$-vector spaces.  Then for any exact tensor functor $\eta$ in the set on the right hand side, its composition with the forgetful functor to $K$-Vect gives an exact tensor functor
	\[F_G: \mathrm{Rep}_K(G)\xrightarrow{\eta} \mathrm{Rep}_K(\pi)\xrightarrow{For} K\text{-Vect}\cong \mathrm{Bun}(\Spa(K)_\et).\]
	By \Cref{p:tannakian-g-torsors-on-sousperf}, this amounts to a $G$-bundle $E$ on $\Spa(K)$.
	
	Since $E$ is necessarily trivial, any element of $\pi$ gives an automorphism of this tensor functor and hence an automorphism of $E$.
	
	This defines a map of sets
	\[f: \pi\rightarrow \mathrm{Aut}_G(E).\]
	Since the group structure on $\pi$ is compatible with compositions of such automorphisms, $f$ is in fact a group homomorphism.
	Since $E$ is trivial, we have $\mathrm{Aut}_G(E)\cong G$, but there is no canonical such isomorphism. Instead, we can regard  $G':=\mathrm{Aut}_G(E)$ as an inner form of $G$. But note that we have a natural equivalence of groupoids
	\[ \Big\{\text{homomorphisms $\pi\to G{(K)}$}\Big\}\isomarrow \Big\{\begin{array}{@{}c@{}l}\text{homomorphisms $\pi\to G'{(K)}$}\\\text{into inner forms $G'$ of $G$}\end{array}\Big\}\]
	as the natural functor is clearly fully faithful and essentially surjective. 
	All in all, we have thus defined functors
	\[\Big\{\begin{array}{@{}c@{}l}\text{cts.\ homomorphisms}\\ \pi\to G{(K)}\end{array}\Big\}\to \Big\{\begin{array}{@{}c@{}l}\text{exact tensor functors from}\\
		\mathrm{Rep}_K(G)\text{ to } \mathrm{Rep}_K(\pi)\end{array}\Big\}\xrightarrow{F_G} \Big\{\text{homomorphisms $\pi\to G{(K)}$}\Big\}.\]
	The construction is clearly natural in $G$. We can therefore now choose any closed embedding $\rho:G^\text{alg}\hookrightarrow \GL_n$ to compare ${F_G}$ with $F_{\GL_n}$. For $\GL_n$, it is clear that the first arrow is an equivalence, with quasi-inverse given by evaluating at the identity $\GL_n\to \GL_n$. Thus the essential image of $F_{\GL_n}$ is given by the continuous homomorphisms. Since $\rho$ is a closed immersion, its $K$-points $\rho(K):G(K)\hookrightarrow\GL_n(K)$ exhibit $G(K)$ as a closed subspace of $\GL_n(K)$, and it follows that a homomorphism $\pi\to G(K)$ is continuous if and only if the composition $\pi\to G(K)\hookrightarrow \GL_n(K)$ is. By naturality in $G$, this shows that the essential image of $F_G$ is given by the continuous homomorphisms. 
	
	That the other composition of the above functors is also equivalent to the identity easily follows from the observation that $\eta$ can be reconstructed from the data of $F_G$ and $f$.
\end{proof}

We have the following rigid analytic analog of the discussion in \cite[\S6, p.68]{Simpson-local-system}:
\begin{Proposition}\label{p:Tannaka-HB}
	Let $X$ be a smooth rigid space. Then there is an equivalence of categories
	\[
	\Big\{
	\text{$G$-Higgs bundles on $X_{\et}$}\Big\}\cong  \Big\{\begin{array}{@{}c@{}l}\text{exact tensor functors}\\
		\mathrm{Rep}_K(G)\to \mathrm{Higgs}(X)\end{array}\Big\}.
	\]
	where $\mathrm{Higgs}(X)$ is the tensor category of Higgs bundle on $X$.
\end{Proposition}
\begin{proof}
	There is clearly a functor from left to right defined by pushout of any given $G$-Higgs bundle along representations $G\to \GL_n$. To construct an inverse, since both sides satisfy \'etale descent, we may localise and assume that $X=\Spa(R)$ is affinoid  and $\Omega_X^1$ is finite free, with a chosen basis $\delta_1,\dots,\delta_d$.
	Let $\mathcal E$ be an object on the right hand side and consider the composition \[\mathcal F:\mathrm{Rep}(G)\xrightarrow{\mathcal E}\mathrm{Higgs}(X)\to \mathrm{Bun}(X_\et)\]
	with the forgetful functor. By \Cref{p:tannakian-g-torsors-on-sousperf}, this defines a $G$-bundle $F$ on $X_{\et}$. 
	
	The datum of the Higgs field given by $\mathcal E$
	amounts to assigning to each $\delta_i$ an endomorphism $\theta_{i,V}:\mathcal F(V)\to \mathcal F(V)$ for any representation $G\to \GL(V)$, functorial in $V$ and compatible with tensor products, such that $\theta_{i,V}$ and $\theta_{j,V}$ commute for all $i,j$ and all $V$. This is the same as the datum of  commutative endomorphisms  $\theta_i:\mathcal F\to \mathcal F$ of the tensor functor $\mathcal F$. We can identify these endomorphisms with automorphisms $1+\theta_i\varepsilon$ over $R[\varepsilon]/\varepsilon^2$. Using that automorphisms of $\mathcal F$ correspond to elements of $G':=\mathrm{Aut}_G(F)$, we see that such data correspond to elements in the kernel of the reduction map $G'(R[\varepsilon]/\varepsilon^2)\to G'(R)$, hence to elements of the Lie algebra of $G'=\mathrm{Aut}_G(F)$. One easily  checks that this Lie algebra is $\Lie(G')= F\times^G \mathfrak g=\mathrm{ad}(F)$. This shows that we may regard $\theta_1,\dots,\theta_d$ as commutative elements of $\wtOm\otimes \mathrm{ad}(F)$, which is the same as a $G$-Higgs field on $F$.
\end{proof}

\begin{Definition}
	We denote by  $\Higgs^{\profet}(X)$  the category of pro-finite-\'etale Higgs bundles on $X$, and by $\Higgs^{\profet}_G(X)$  the category of pro-finite-\'etale $G$-Higgs bundles on $X_{\et}$.
\end{Definition} 

The following is then the rigid analytic analog of Simpson's Tannakian description of the category of semi-stable Higgs bundles with vanishing Chern classes:

\begin{Lemma}\label{l:Tannaka}
	Let $X$ be any connected smooth proper rigid space for which  $\widetilde{X}$ from \Cref{d:univ-cover} is sousperfectoid, e.g.\ $X$ could be an abeloid variety. Then there is an equivalence of categories
	\begin{align*}
		\Higgs^{\profet}_G(X)\to \Big\{\begin{array}{@{}c@{}l}\text{exact tensor functors}\\
			\mathrm{Rep}_K(G)\to \Higgs^{\profet}(X) \end{array}\Big\}.
	\end{align*}
\end{Lemma}
\begin{proof}  
	It suffices to prove that the pro-finite-\'etale assumptions on both sides are identified under the equivalence in  \Cref{p:Tannaka-HB}. It thus suffices to prove the statement for $G$-Higgs bundles replaced by $G$-bundles.
	For this we use \Cref{p:tannakian-g-torsors-on-sousperf} for both $X$ and $\widetilde{X}$. By functoriality, this says that the morphism $\widetilde{X}\to X$ induces a commutative diagram:
	\[\begin{tikzcd}[row sep=0.2cm]
		{\Big\{\text{$G$-bundles on $X_{\et}$}\Big\}} \arrow[r,"\sim"] \arrow[d] & {\Big\{\begin{array}{@{}c@{}l}\text{exact tensor functors}
				\\\mathrm{Rep}_K(G)\to \mathrm{Bun}(X_{\et})\end{array}\Big\}}\arrow[d] \\
		{\Big\{\text{$G$-bundles on $\widetilde X_{\et}$}\Big\}} \arrow[r,"\sim"]           & {\Big\{\begin{array}{@{}c@{}l}\text{exact tensor functors}
				\\\mathrm{Rep}_K(G)\to \mathrm{Bun}(\widetilde X_{\et})\end{array}\Big\}}          
	\end{tikzcd}\]
	Taking essential preimages of the identity objects in the bottom row, we deduce that the top row identifies pro-finite-\'etale $G$-bundles with tensor functors into pro-finite-\'etale objects in $\mathrm{Bun}(X_{\et})$, as we wanted to see.
	
\end{proof}

We now have everything in place to prove \Cref{t:tannakian}:
\begin{proof}[Proof of \Cref{t:tannakian}]
When $G^\text{alg}=\mathrm{GL}_n$ for some $n\in \N$, the theorem holds by \cite[Theorem~1.1]{HMW-abeloid-Simpson}, which moreover says that we have an exact tensor equivalence 
\[ \mathrm{Rep}_{K}(\pi_1^{\et}(X,0))=\Higgs^{\profet}(X).\]
The general case now follows from the combination of \Cref{p:pirep} and \Cref{l:Tannaka}.
\end{proof}

Finally, let us mention that the Tannakian formalism also applies to $G$-torsors on $X_\mathrm{v}$:
\begin{Proposition}\label{p:v-bundles-Tannakian}
Let $X$ be any diamond over $K$, then there is an equivalence of categories
\[
\Big\{
\text{$G$-bundles on $X_{\mathrm{v}}$}\Big\}\cong  \Big\{\begin{array}{@{}c@{}l}\text{exact tensor functors}\\
	\mathrm{Rep}_K(G)\to \mathrm{Bun}(X_\mathrm{v})\end{array}\Big\}.
\]
where $\mathrm{Bun}(X_\mathrm{v})$ is the category of v-vector bundles on $X$.
\end{Proposition}
\begin{proof}
It is clear that we have a functor from left to right defined by pushout. To see that this is an equivalence, we use that by definition, any diamond admits a v-cover $\pi:Y\to X$ such that $Y\times_XY$ is again perfectoid. 

Let $F:\mathrm{Rep}_K(G)\to \mathrm{Bun}(X_\mathrm{v})$ be an exact tensor functor. Then 
\[ \pi^{\ast}F:\mathrm{Rep}_K(G)\to \mathrm{Bun}(X_\mathrm{v})\to \mathrm{Bun}(Y_\mathrm{v})\]
is still an exact tensor functor. Since $Y$ is perfectoid, we have $\mathrm{Bun}(Y_\mathrm{v})=\mathrm{Bun}(Y_\et)$ by \cite[Theorem 3.5.8]{KedlayaLiu-II} (see also \cite[Lemma 17.1.8]{berkeley}). We can therefore invoke \Cref{p:tannakian-g-torsors-on-sousperf} to see that $\pi^\ast F$ corresponds to a $G$-bundle $\mathcal E$ on $Y$. By functoriality of \Cref{p:tannakian-g-torsors-on-sousperf} and naturality in the sousperfectoid space, the canonical isomorphism over $Y\times_XY$
\[\psi:p_1^{\ast}\pi^{\ast}F\isomarrow p_2^{\ast}\pi^{\ast}F\]
corresponds to an isomorphism $\psi_{\mathcal E}:p_1^{\ast}\mathcal E\isomarrow p_2^{\ast}\mathcal E$. Similarly, the cocycle condition for $\psi$ ensures the cocycle condition for $\psi_{\mathcal E}$ on $Y\times_XY\times_XY$. Hence $\psi_{\mathcal E}$ defines a descent datum for a $G$-torsor along the v-cover $Y\to X$. This defines the desired $G$-bundle on $X_\mathrm{v}$.
\end{proof}

As an application, this allows us to answer \Cref{q1} for linear algebraic groups: Namely, it  shows that Faltings' $p$-adic Simpson correspondence for proper smooth curves \cite{Faltings_SimpsonI} generalises to linear algebraic groups. More generally, the same is true for the $p$-adic Simpson correspondence of smooth proper rigid spaces of \cite{heuer-paSimpson}:
\begin{Theorem}
Let $X$ be a smooth proper rigid space over $K$. Let $G$ be any linear algebraic group over $K$.
Then choices of a flat $B_{\dR}^+/\xi^2$-lift $\mathbb X$ of $X$ and of an exponential for $K$ induce an exact tensor equivalence
\[
\Big\{
\text{$G$-bundles on $X_{v}$}\Big\}\isomarrow  \Big\{
\text{$G$-Higgs bundles on $X_{\et}$}\Big\}.
\]
\end{Theorem}
\begin{proof}
The category of $G$-torsors  on $X_v$ is equivalent to the category of $G$-torsors on $X_{\proet}$ by \cite[Corollary~1.2]{heuer-G-Torsor}. 
For $G=\GL_n$, the result therefore holds by \cite[Theorem~5.1]{heuer-paSimpson}, which moreover says that in this case, the correspondence is an exact tensor functor. The case of general $G$ follows by the Tannakian formalism using \Cref{p:v-bundles-Tannakian} and \Cref{p:Tannaka-HB}.
\end{proof}
We have thus answered \Cref{q1} for linear algebraic $G$, and \Cref{q3} for linear algebraic $G$ on abeloid varieties.

\section{The small correspondence}
In this section, the aspect of the $p$-adic Simpson correspondence for which we explore generalizations to general rigid groups $G$ is the ``small" correspondence in the ``global'' setup. That is, we allow $X$ to be any smooth adic space, not necessarily proper or toric. 

The starting point of our discussion is the local $p$-adic Simpson correspondence for small $G$-bundles constructed in \cite[Theorem~6.5]{heuer-Moduli}. We first explain in what sense this is functorial with respect to perfectoid abelian Galois covers (c.f. \Cref{d: equivariant map} below). While it is in general impossible to find such perfectoid covers globally, we explain at the example of abelian varieties how one can use this to contruct instances of a small correspondence for $G$-bundles in the global setting.

In order to explain the setup, we first recall that in Faltings' setting \cite{Faltings_SimpsonI}, the local $p$-adic Simpson correspondence relies on the following notion:

\begin{Definition}
Let $U$ be an affinoid smooth rigid space. A toric chart of $U$ is an \'etale morphism $h:U\to \mathbb T^d$ to some affinoid torus $\mathbb{T}^d = \Spa(K\langle T^{\pm 1}_1,\dots, T^{\pm 1}_d\rangle)$ which is a composition of rational opens and finite \'etale maps. We call $U$ toric if a toric chart exists. 
\end{Definition}

The setup of our small correspondence can be divided into a local and a global one:
\begin{Setup}[Local Setup]\label{a:local-cover}
Let $U$ be a toric smooth rigid space. We assume that
\[f:U_\infty\to U\] is a pro-\'etale Galois cover in Scholze's pro-\'etale site $U_{\mathrm{pro}\et}$ of \cite[\S3]{Scholze-padic-Hodge} satisfying the following:
\begin{enumerate}
	\item $U_\infty$ is affinoid perfectoid.
	\item  The covering group $\Delta_U$ of $f$ is isomorphic as a profinite group to $\Z_p^d$ for $d:=\dim U$.
	\item The following natural Cartan--Leray map associated to $f$ is an isomorphism:
	\[ \Hom_{\cts}(\Delta_U,\O(U))\to H^1_{\cts}(\Delta_U,\O(U_\infty))\to H^1_\mathrm{v}(U,\O).\]
\end{enumerate}
\end{Setup} 

Here the first morphism is induced by $\Hom_{\cts}(\Delta_U,\O(U))=H^1_{\cts}(\Delta_U,\O(U))$, and the second map  is always an isomorphism due to the assumption that $U_\infty$ is affinoid perfectoid.
\begin{Example} Given a toric chart $h:U\to \mathbb T^d$, pullback along $h$ of the affinoid perfectoid torus $\mathbb T^d_\infty\to \mathbb T^d$  induces a pro-\'etale cover $U_\infty \to U$. This satisfies \Cref{a:local-cover} by \cite[Lemmas~4.5, 5.5]{Scholze-padic-Hodge}. That being said, we note that we only assume in \Cref{a:local-cover} that $U$ is toric, but we do not assume that $f$ is related to any chart.
\end{Example}

\begin{Setup}[Global Setup]\label{a:global-cover}
Let $X$ be a smooth rigid space. We assume that  \[f:X_\infty\to X\]
is  a perfectoid pro-\'etale Galois cover in $X_\proet$ with covering group $\Delta_X$ for which there is a basis of $X_{\et}$ given by affinoids $U$ on which $U_\infty:=X_\infty\times_{X}U\to U$ satisfies \Cref{a:local-cover}. We call such morphisms $f$ quasi-toric.
\end{Setup}
\begin{Example}
If $X$ is an abeloid variety of dimension $d$, then the cover $\widetilde X=\varprojlim_{[p]} X\to X$ is a perfectoid $T_pX$-torsor \cite[Theorem~1]{AWS-perfectoid}. While $T_pX\cong \Z_p^{2d}$ does not have the correct rank to satisfy \Cref{a:global-cover}, we have the following result in case of ordinary reduction, i.e.\ when the \'etale part $H^{\et}$ of the $p$-divisible group $H$ of the formal Raynaud extension associated to $X$ (the formal group scheme whose generic fibre is the maximal semi-abelian subgroup of $X$ of good reduction) is $\cong(\mathbb{Q}_p / \mathbb{Z}_p)^d$.  In this case we call any subgroup of $X[p^\infty]$ anticanonical if it is isomorphic to its image in the generic fibre of $H^\et$. 
\end{Example}

\begin{Lemma}\label{l:abeloid-assumption}
Let $X$ be an ordinary abeloid variety. Fix any anticanonical	
$p$-divisible subgroup $D\subseteq X[p^\infty]$ of rank $d$. Its Tate module $T_pD\subseteq T_pX$ is a submodule which corresponds to a $\Delta_X:=T_pX/T_pD$-torsor $X_\infty \to X$. This satisfies the conditions of \Cref{a:global-cover}.
\end{Lemma}

\begin{proof}
We first explain why $X_\infty$ is perfectoid: If $X$ has good reduction, this can be seen as in \cite[Lemma 2.13]{AWS-perfectoid}: Let $\mathfrak X_n$ be the formal model of $X/D_n$. The transition map $\mathfrak X_n\to \mathfrak X_{n-1}$ is the quotient by the canonical subgroup, which reduces to the relative Frobenius mod $p$. Thus Frobenius is surjective on the mod $p$ fibre of $\varprojlim_n  \mathfrak X_n$, hence $X_\infty$ is perfectoid.

The general case follows from the one of good reduction via Raynaud uniformisation by the same argument as in \cite[Proposition 4.3]{AWS-perfectoid}, replacing $[p]$ in \cite[diagram (5)]{AWS-perfectoid} with the canonical isogeny $X/D_n\to X/D_{n-1}$. This shows that $X_\infty=\varprojlim_n X/D_n$ is perfectoid.

Since $X$ is ordinary, the Hodge--Tate map admits a factorisation 
\[ \begin{tikzcd}[column sep = 0.1cm,row sep =0.2cm]
	\Hom(T_pX,K)\arrow[rr,"\HT"]\arrow[rd]&& H^0(X,\Omega^1(-1))\\
	& \Hom(T_pC,K)\arrow[ru,dotted]&
\end{tikzcd}\]
where $T_pC\subseteq T_pX$ is the Tate module of the canonical $p$-divisible subgroup, the map on the left is the restriction, and the dotted map is an isomorphism. Since $D$ is anticanonical, the composition $T_pC\to T_pX\to \Delta_X$ is an isomorphism, so it follows that the map
\[ \Hom(\Delta_X,K)\to 	\Hom(T_pX,K)\to H^0(X,\Omega^1(-1))\] 
is an isomorphism. Let now $U\subseteq X$ be any affinoid subspace such that $U_\infty=X_\infty\times_XU$ is affinoid perfectoid. Using that $\Omega_X^1$ is a trivial vector bundle in the case of abelian varieties, it then follows by applying $-\otimes_K\O(U)$ that also the following map is an isomorphism
\[ \Hom(\Delta_X,\O(U))\to 	\Hom(T_pX,\O(U))\to H^0(U,\Omega^1(-1))=H^1_\mathrm{v}(U,\O).\qedhere\]
\end{proof}

\begin{Remark}
With more work, one can in fact show that there exists a cover satisfying \Cref{a:global-cover} for any abelian variety $X$, not necessarily ordinary: After replacing $X$ by an isogenous abelian variety,  we can reduce to the case that $X$ is principally polarized. It then follows from \cite[Corollary III.3.12]{Scholze-torsion} that $X$ is $p$-power isogeneous to an abelian variety $X'$ that admits a canonical subgroup (see \cite[Definition III.2.7]{Scholze-torsion} for the notion of canonical subgroup in this case). For any choice of anti-canonical $p$-divisible subgroup $(D_n\subseteq X'[p^n])_{n\in \N}$, the tower of canonical isogenies $\varprojlim_{n\in\N} X'/D_n$  then has the desired form. But we will not use this fact in the present paper. It is plausible that a cover satisfying \Cref{a:global-cover} exists more generally for abeloid varieties.
\end{Remark}

In order to construct the small $p$-adic Simpson correspondence in the setup of \Cref{a:global-cover}, we now first need to deal with the local situation.

\begin{Lemma}\label{Lemma2.14}
Let $f:U_\infty\to U$ be as in \Cref{a:local-cover}. Then for any $n\geq 1$, there is $\gamma>0$ such that for any $s\in \N$, the kernel and cokernel of the natural map
\[ \iota^n_s:H^n_{\cts}(\Delta_X,\O^+(U)/p^s)\to H^n_{\cts}(\Delta_X,\O^+(U_\infty)/p^s)\]
are $p^\gamma$-torsion. The same is true for $\iota^n:H^n_{\cts}(\Delta_X,\O^+(U))\to H^n_{\cts}(\Delta_X,\O^+(U_\infty))$.
\end{Lemma}
\begin{proof}
The Cartan--Leray sequence for $U_\infty\to U$ induces a commutative diagram
\[\begin{tikzcd}
	H^n_{\cts}(\Delta_X,\O(U))\arrow[r]& H^n_{\cts}(\Delta_X,\O(U_\infty))\arrow[r]&H^n_\mathrm{v}(U,\O)\\
	H^n_{\cts}(\Delta_X,\O^+(U))\arrow[r]\arrow[u]& H^n_{\cts}(\Delta_X,\O^+(U_\infty))\arrow[r]\arrow[u]&H^n_\mathrm{v}(U,\O^+)\arrow[u]
\end{tikzcd}\]
The horizontal maps on the right are (almost) isomorphisms because $R\Gamma_v(U_\infty,\O^+)\stackrel{a}{=} \O^+(U_\infty)$ due to the assumption that $U_\infty$ is affinoid perfectoid. For $n=1$, the composition of the top row is an isomorphism by assumption. It follows that the top map is an isomorphism for general $n$: This is because Cartan--Leray is compatible with cup products in group cohomology, respectively sheaf cohomology, and the cup products induce isomorphisms
\[ 	H^n_{\cts}(\Delta_X,\O(U))=	\wedge^nH^1_{\cts}(\Delta_X,\O(U)),\quad H^n_\mathrm{v}(U,\O)=\wedge^nH^1_\mathrm{v}(U,\O),\]
where the last equality follows from \cite[Lemma~5.5]{Scholze-padic-Hodge} using that $U$ is toric.

To see that the bottom composition is an isomorphism up to bounded $p$-torsion  cokernel, we use again that $U$ admits a toric chart: Considering the Cartan--Leray sequence of the associated toric affinoid cover, this shows that there is an injective map \[\wedge^n\O^+(U)^d\to H^n_\mathrm{v}(U,\O^+)\] with $p^\beta$-torsion cokernel for some $\beta>0$. This shows that $p^{\beta}\cdot \iota^n$ factors through a map $j:	H^n_{\cts}(\Delta_X,\O^+(U))\to \wedge^n\O^+(U)^d$. As $H^n_{\cts}(\Delta_X,\O^+(U))=\wedge^n\O^+(U)$, the map $j$ is a morphism between finite free $\O^+(U)$-modules which is an isomorphism after inverting $p$. Hence it is injective with $p^{\delta}$-torsion cokernel for some $\delta>0$. Setting $\gamma:=\beta+\delta$, we get the desired result for $\iota^n$.
The statement for $\iota^n_s$ follows from this by long exact sequences.
\end{proof}

Let now $G$ be any rigid group over $K$. Following the notation in \S\ref{s:local-struct-G}, we fix an open subgroup $G_0$ of $G$ of good reduction and let $k >0$ be such that that the Lie algebra exponential $\exp:\mathfrak g_k\to G_k$ converges and is an isomorphism. In this setting, we can now generalise the technical notion of smallness from \cite[\S6]{heuer-Moduli}:
\begin{Definition}\label{d:smallness}\label{d: small-G-Higgs}
Let $U_\infty \to U$ be as in \Cref{a:local-cover} and let $\gamma>0$ be as in \Cref{Lemma2.14}. Set $c:= 10\cdot \max(\gamma,k)$.
\begin{enumerate}
	\item We call a continuous representation $\rho:\Delta_U\to G(U)$ an $f$-small cocycle if it factors through $G_c(U)$.
	We denote the category of $f$-small $1$-cocycles by $\mathcal Z^{1,\sm}_\mathrm{cts}(\Delta_U, G(U))$, where the morphisms $\rho_1\to \rho_2$ are given by $\{g\in G(U)|g \rho_1g^{-1}=\rho_2\}$.
	\item A $G$-bundle on $U_\mathrm{v}$ is said to be $f$-small if it admits a reduction of structure group to $G_c$. We denote the category of $f$-small $G$-bundles on $U_\mathrm{v}$ by $\Bun_G^{\sm}(U_\mathrm{v})$.
	\item A $G$-Higgs bundle $(E,\theta)$ on $U_\text{\'et}$ is said to be $f$-small if $E$ is trivial and there is a trivialisation $E\cong G$ with respect to which $\theta$ is in the image of the chain of maps
	\begin{align*}
		\HTlog\colon&\Homc(\Delta_U, G_{c}(U)) \xrightarrow{\mathrm{log}}
		\Homc(\Delta_U, \lie_c(U))
		\isomarrow \Hc(\Delta_U, \lie_c(U))\\
		&\rightarrow \Hc(\Delta_U,\lie_c(U_\infty)) \to \Hv(U, \lie_c)\xrightarrow{\HT} R^1\nu_\ast\lie_c(U)\isomarrow H^0(U,\lie_c\otimes \widetilde \Omega).
	\end{align*}
	We denote the category of $f$-small $G$-Higgs bundles on $U$ by $\HiggsG^{\sm}(U)$. 
\end{enumerate}
\end{Definition}
\begin{Remark}
\begin{enumerate}
	\item   We emphasize that one way in which these notions depend on $f$ is that $c$ depends on $f$.
	Note that we have not characterised the $\gamma$ in \Cref{Lemma2.14} uniquely. This is on purpose, as we need some leeway to increase $c$ later on. It would be more precise to call the above notions ``$f$-$c$-small'', but we shall drop this for simplicity.
	\item
	The reason for the factor 10 in the definition of $c$ ultimately lies in the appearance of the same factor in \cite[Remark~6.8]{heuer-Moduli}: The constant $5c$ appears in \cite[Proposition~5.5]{heuer-Moduli}, and an additional factor of 2 is picked up in the globalisation step as explained in detail in \cite[Remark~6.8]{heuer-Moduli}. That being said, we are not aiming for an optimal value of $c$. For most purposes, the important point is simply that $c$ only depends on $\gamma$.
\end{enumerate}
\end{Remark}
The three notions of \Cref{d:smallness} are related as follows:
\begin{Lemma}\label{l:LS-functors}
\begin{enumerate}
	\item Pushout of the $\Delta_U$-torsor $f:U_\infty\to U$ along any $f$-small $1$-cocycle defines a functor
	\[ \mathcal Z^{1,\sm}_\mathrm{cts}(\Delta_U, G(U))\to \Bun_G^{\sm}(U_\mathrm{v}).\]
	\item Sending an $f$-small $1$-cocycle $\rho$ to $(G,\HTlog(\rho))$ defines a functor
	\[ \mathcal Z^{1,\sm}_\mathrm{cts}(\Delta_U, G(U))\to \Higgs_G^{\sm}(U). \]
\end{enumerate}
\end{Lemma}
\begin{proof}
Part (1) is clear from the definitions: Note that any morphism between $f$-small \mbox{$1$-cocycles} induces an isomorphism between the attached v-$G$-bundles. For part (2), we just need to verify that $\theta:=\HTlog(\rho)$ is a Higgs field:
Since $\Delta_U$ is commutative, it follows from \cite[Lemma~3.10]{heuer-G-Torsor} that the image of $\rho$ under $\log$ in the above sequence lies in $\Homc(\Delta_U, \lie_c(U))^{[-,-]=0}$. Going through the almost isomorphisms, it is clear from this that its image $\theta$ in $H^0(U,\lie_c\otimes \widetilde \Omega)$ is a $G$-Higgs field.
\end{proof}

\begin{Proposition}\label{p: small-correspondence}
Let $U_\infty\to U$ be as in \Cref{a:local-cover}. Let $G$ be a rigid group with setup as in \Cref{d:smallness}. Then the functors of \Cref{l:LS-functors} define equivalences of categories
\[\Bun^{\sm}_G(U_\mathrm{v}) {\,\xleftarrow{\sim}\,} 
\mathcal Z^{1,\sm}_\mathrm{cts}(\Delta_U, G(U))
{\,\xrightarrow{\sim}\,}
\HiggsG^\sm(U)
\]
\end{Proposition}
\begin{proof}
This is part of \cite[Theorem~6.5, Remark 6.8]{heuer-Moduli}: More precisely, the axiomatic conditions detailed in  \cite[Remark 6.8]{heuer-Moduli} are satisfied in \Cref{a:local-cover} by \Cref{Lemma2.14}.
\end{proof}
\begin{Definition}
We call the composition \[
\mathrm{LS}_{f,G}:\Bun^{\sm}_G(U_\mathrm{v})\isomarrow \HiggsG^{sm}(U)\]
the small $p$-adic Simpson correspondence on $U$ attached to $f$.
\end{Definition}

Next, we explain the naturality of $\LS_f$ in $f$. For this we need morphisms of Galois covers:
\begin{Definition}
\label{d: equivariant map}
Let $f:X_\infty\rightarrow X$ and $g: Y_\infty\rightarrow Y$ be quasi-toric coverings with Galois groups $\Delta_X$ and $\Delta_Y$. A \textit{Galois-equivariant map} $\beta:f\Rightarrow g$ is a triple $\beta=(\varphi, u, \tilde{u})$ consisting of a group homomorphism $\varphi: \Delta_X\rightarrow \Delta_Y$ and a commutative diagram
\[\begin{tikzcd}
	X_\infty \ar[r,"\tilde{u}"] \ar[d,"f"] & Y_\infty \ar[d,"g"]\\
	X \ar[r,"u"]&  Y,
\end{tikzcd}\]
such that $\tilde{u}$ is $\Delta_X$-equivariant with respect to the $\Delta_X$-action on $Y_\infty$ via $\varphi$.
\end{Definition}
\begin{Proposition}\label{p: functorial}
Let $f:X_\infty\rightarrow X$, $g: Y_\infty\rightarrow Y$ and $h: Z_\infty\rightarrow Z$ be quasi-toric covers  and let $\beta=(\varphi, u, \tilde{u}):f\Rightarrow g$ and $\beta'=(\varphi', u', \tilde{u}'):g\Rightarrow h$ be Galois equivariant morphisms. Let $G$ be a rigid group and let $\gamma>0$ be large enough such that \Cref{Lemma2.14} holds for $f$, $g$ and $h$.
\begin{enumerate}
	\item
	The equivalence $\LS_{f,G}$ is natural in $f$: There is a natural isomorphism of functors $\alpha_\beta:\LS_f\circ u^\ast \Rightarrow  u^\ast\circ \LS_g$ making the following diagram 2-commutative:
	\[
	\begin{tikzcd}\Bun^{\sm}_G(X_\mathrm{v}) \arrow[r,"\LS_f"] \arrow[dr,"\alpha_{\beta}",Rightarrow,shorten >=3ex,shorten <=3.0ex] & \HiggsG^{\sm}(X)      \\
		\Bun^{\sm}_G(Y_\mathrm{v}) \arrow[u, "u^\ast"] \arrow[r,"\LS_g"]  & \HiggsG^{\sm}(Y)  \arrow[u, "u^\ast"]. \end{tikzcd}\]
	\item $\alpha$ is compatible with composition: For the Galois-equivariant map $\beta'\circ \beta$, we have
	\[\alpha_{\beta'\circ \beta}= u^\ast\alpha_{\beta'} \circ  \alpha_{\beta}.\]
\end{enumerate}
\end{Proposition}

\begin{proof} For (1), we need to produce naturals isomorphisms making the diagram
\[
\begin{tikzcd}
	\LS_f:\Bun^{\sm}_G(X_\mathrm{v})&  {\mathcal Z^{1,\sm}_\cts(\Delta_X, G(X))} \arrow[l] \arrow[r] & \HiggsG^{\sm}(X)      \\
	\LS_g:\Bun^{\sm}_G(Y_\mathrm{v}) \arrow[u, "u^\ast"]        & {\mathcal Z^{1,\sm}_\cts(\Delta_Y, G(Y))}\arrow[u, "u^\ast \varphi^\ast"]\arrow[l]\arrow[r] & \HiggsG^{\sm}(Y)  \arrow[u, "u^\ast"] 
\end{tikzcd}
\]
2-commutative.
For the commutativity of the left square, starting from a $1$-cocycle $\rho: \Delta_Y\rightarrow G(Y)$, we need to find a natural isomorphism between the two $G$-bundles on $X_\mathrm{v}$
\[\quad X_\infty \times^{\Delta_X} G\quad \text{and} \quad X\times_Y (Y_\infty \times^{\Delta_Y} G)\]
where to form the second contracted product we push out $g: Y_\infty\rightarrow Y$ using $\rho$, and for the first, we push out $f: X_\infty\rightarrow X$ along the composition
\[\Delta_X\xrightarrow{\varphi} \Delta_Y\xrightarrow{\rho} G(Y)\xrightarrow{u} G(X).\]
This is the functoriality of the Cartan--Leray construction: Presenting each of these two $G$-bundles as quotients, we have by equivariance and universal properties a natural morphism
\[(X_\infty\times G)/\Delta_X\to  (X\times_YY_\infty\times G)/\Delta_Y\isomarrow X\times_Y(Y_\infty\times G)/\Delta_Y\]
of $G$-torsors, which is necessarily an isomorphism.
%\[
%\begin{tikzcd}
%X_\infty\times G\times \underline{\Delta}_X \arrow[r, "p"' swap,  ,yshift=0.5ex] \arrow[r, "a"', yshift=-0.5ex] \arrow[d, %"{(f\times\tilde{u})\times id \times \beta}"'] & X_\infty\times G \arrow[d, "{(f\times\tilde{u})\times id}"'] \arrow[r, '] %& X_\infty\times^{\Delta_X} G \arrow[d, ', dashed]\\
%X\times_Y Y_\infty\times G \times \underline{\Delta}_Y \arrow[r,  "p"' swap, yshift=0.5ex] \arrow[r, "a"',yshift=-0.5ex] & %X\times_Y Y_\infty\times G \arrow[r, '] & X\times_Y (Y_\infty \times^{\Delta_Y} G).
%\end{tikzcd}
%\]
%where $p$ is the projection map and $a$ is the diagonal action of $\underline{\Delta}_X$ on $X_\infty\times G$, resp. that of $\underline{\Delta}_Y$ on $X\times_YY_\infty\times G$. This induces the dashed map of $G$-torsors over $X$, which is necessarily an isomorphism. The naturality is clear.

For commutativity of the right square, it suffices to check commutativity of the diagram
\[
\begin{tikzcd}
	\Homc(\Delta_Y, G_k(Y)) \ar[r,"\HTlog"]\ar[d,"u^\ast"] & H^0(Y,\lie_c\otimes \widetilde \Omega)\ar[d,"u^\ast\varphi^\ast"]\\
	\Homc(\Delta_X, G_k(X)) \ar[r,"\HTlog"] & H^0(X,\lie_c\otimes \widetilde \Omega)
\end{tikzcd}
\]
where  $\HTlog$ was defined in \Cref{d: small-G-Higgs}. This is immediate from functoriality of $\log$ and the functoriality of the Cartan--Leray map applied to $\mathfrak g_c$. This concludes part (1).

For part (2), it is clear that $(\beta'\circ \beta, u'\circ u, \tilde{u}'\circ \tilde{u})$ is a Galois-equivariant map, and it suffices to prove the composition property for the left and right squares in the above diagram. For the right hand side, this is clear. For the left, by construction, the natural transformation $\alpha_{\beta'\circ\beta}$ comes from 
\[(X_\infty\times G)/\Delta_X\isomarrow X\times_Z(Z_\infty\times G)/\Delta_Z\]
%induced by the commutative diagram
%\[
%\begin{tikzcd}
%X_\infty\times G\times \underline{\Delta}_X \arrow[r, "p"' swap,yshift=0.5ex] \arrow[r, "a"',yshift=-0.5ex] \arrow[d, "{(f\times\tilde{v}\circ\tilde{u})\times id \times \beta_2\circ\beta_1}"'] & X_\infty\times G \arrow[d, "{(f\times\tilde{v}\circ\tilde{u})\times id}"'] \arrow[r, '] & X_\infty\times^{\Delta_X} G \arrow[d, ', dashed]\\
%X\times_Z Z_\infty\times G \times \underline{\Delta}_Z \arrow[r,  "p"' swap,yshift=0.5ex] \arrow[r, "a"',yshift=-0.5ex] & X\times_Z Z _\infty\times G \arrow[r, '] & X\times_Z (Z _\infty \times^{\Delta_Z } G),
%\end{tikzcd}
%\]
whereas $\alpha_{\beta'}\circ\alpha_\beta$ comes from the composition
\[(X_\infty\times G)/\Delta_X\isomarrow X\times_Y(Y_\infty\times G)/\Delta_Y\isomarrow X\times_Z(Z_\infty\times G)/\Delta_Z,\]
%obtained by the composition of diagrams 
%\[
%\begin{tikzcd}[column sep=2cm]
%X_\infty\times G\times \underline{\Delta}_X \ar[d, shift right,"p" swap]\ar[d, shift left,"a"]
%\ar[r,"{(f\times\tilde{u})\times id \times \beta_1}"] 
%& X\times_Y Y_\infty\times G \times \underline{\Delta}_Y  \ar[d, shift right, "p" swap] \ar[d, shift left,"a"] \ar[r,"id%\times \tilde{v}\times id\times \beta_2"]
%& X\times_Z Z_\infty\times G \times \underline{\Delta}_Z  \ar[d, shift right, "p" swap] \ar[d, shift left,"a"]\\
%X_\infty\times G \ar[r,"{(f\times\tilde{u})\times id}"] & X\times_Y  Y_\infty\times G	\ar[r,"id\times v\times id"]& X\times_Z  Z_\infty\times G.
%\end{tikzcd}
%\] 
It is clear from the universal properties  that these two morphisms agree.
\end{proof}

Given this naturality, we can now glue the local small $p$-adic Simpson correspondences:

\begin{Definition}\label{d:loc-small}
Let $X$ be a quasi-compact smooth rigid space. Let $f:X_\infty\to X$ be a quasi-toric cover as in \Cref{a:global-cover}.
\begin{enumerate}
	\item A $G$-bundle $V$ on $X_\mathrm{v}$ is called (\'etale) locally small with respect to $f$ if there is an \'etale cover  $h:\wt X\rightarrow X$ by an affinoid $\wt X$ as in \Cref{a:local-cover} such that $h_{2}:\wt X\times_X \wt X\rightarrow X$ and $h_{3}:\wt X\times_X \wt X\times_X \wt X\rightarrow X$ are also as in \Cref{a:local-cover}, and such that $h^\ast V$ is small with respect to $\wt f: X_\infty\times_X \wt X\rightarrow \wt X$, and similarly $h_{2}^\ast V$ and $h_{3}^\ast V$ are also small. We denote the category of locally small $G$-bundles on $X_\mathrm{v}$ by $\Bun_G^{\lsm}(X_\mathrm{v})$.
	\item A $G$-Higgs bundle $(E, \theta)$ on $X_\et$ is called (\'etale) locally small with respect to $f$ if there is an \'etale cover  $\wt X\to X$ by an affinoid as in (1), such that $h^{\ast}(E, \theta)$ is small with respect to $\wt f$ and $h_{2}^\ast (E, \theta)$ and $h_{3}^\ast (E, \theta)$ are also small in this sense. We denote the category of locally small $G$-Higgs bundles on $X$ by $\Higgs_G^{\lsm}(X)$. \end{enumerate}
	\end{Definition}
	
	\begin{Theorem}\label{p: Global-small-correspondence}
Let $X$ be a smooth rigid space over $K$ with a Galois cover $f$ as in \Cref{a:global-cover}. Then the local equivalences with respect to $f$ glue to an equivalence of categories
\[\mathrm{S}^{\lsm}_f:\Bun_G^{\lsm}(X_\mathrm{v})\to \Higgs_G^{\lsm}(X).\]
In particular, we obtain such a correspondence for ordinary abeloid varieties as in \Cref{l:abeloid-assumption}.
\end{Theorem}
\begin{proof}
This is a  formal consequence of \Cref{p: small-correspondence} and \Cref{p: functorial}. As descent arguments in the context of the $p$-adic Simpson correspondences are subtle in general, we now verify this in detail: Let $V$ be a locally small $G$-bundle on $X_\mathrm{v}$ and let $h:\wt X\to X$ be a cover as in \Cref{d:loc-small}.(1). Then by \Cref{p: small-correspondence} we have an equivalence
\[\Psi:= \LS_{\wt f}: \Bun_G^{\sm}(\widetilde{X}_\mathrm{v})\xrightarrow{\sim} \HiggsG^{\sm}(\widetilde{X}).\]
Let $\wt f_2$ be the pullback $\wt X\times_X\wt X\times_XX_\infty\to \wt X\times_X\wt X$ and similarly $\wt f_3$ the pullback to $\widetilde{X}\times_X \widetilde{X}\times_X \widetilde{X}$, then by assumption we also get
$\Psi_2:= \LS_{\wt f_2}$ on $\wt X\times_X\wt X$ and $\Psi_3:= \LS_{\wt f_3}$.

Consider the \v{C}ech nerve of the cover $\widetilde{X} \rightarrow X$. We write the low degree projections as
\[ \widetilde{X}\times_X\widetilde{X} \times_X\widetilde{X}\xrightarrow{p_{ij}} \widetilde{X}\times_X\widetilde{X}\xrightarrow{p_{k}} \wt X\]
and write $\pi_i$ for the $i$-th projection $\widetilde{X}\times_X\widetilde{X}\times_X\widetilde{X}\to \wt X$. Hence $\pi_1=p_1\circ p_{12}$ and so on. The maps $\wt f$, $\wt f_2$ and $\wt f_3$ induce Galois equivariant maps of quasi-toric covers over these. We denote the resulting isomorphisms $\alpha$ from \Cref{p: functorial} induced by $p_i$, $p_{ij}$, $\pi_i$ by  the corresponding subscripts, i.e. $\alpha_{p_1}$ etc.

The pullback $\wt V:=h^{\ast}V$ carries a natural descent datum $\varphi: p_1^\ast \widetilde{V}\isomarrow p_2^\ast \widetilde{V}$, satisfying the cocycle condition $ p_{23}^\ast \varphi \circ p_{12}^\ast \varphi = p_{13}^\ast \varphi$. We need to show that the isomorphism
\[\psi: p_1^\ast \Psi(\widetilde{V})\xrightarrow{\overset{\alpha_{p_1}^{-1}}{\sim}}  \Psi_2(p_1^\ast \widetilde{V}) \xrightarrow{\overset{\Psi_2(\varphi)}{\sim}}  \Psi_2(p_2^\ast \widetilde{V}) \xrightarrow{\overset{\alpha_{p_2}}{\sim}}  p_2^\ast \Psi(\widetilde{V}) \]
satisfies the cocycle condition
$p_{23}^\ast \psi \circ p_{12}^\ast \psi = p_{13}^\ast \psi$
and hence provides a descent datum for $\Psi(\widetilde{V})$ along $\widetilde{X}\rightarrow X$.
For this we first claim that the following diagram is commutative
\[\begin{tikzcd}
	\Psi_3(\pi_1^\ast\widetilde{V})\ar[rrr,"\Psi_3(p_{12}^\ast\varphi)"]\ar[d,"\alpha_{\pi_1}"]  &&& \Psi_3(\pi_2^\ast\widetilde{V})\ar[d, "\alpha_{\pi_2}"]\\
	\pi_1^\ast \Psi(\widetilde{V})\ar[r,"p_{12}^\ast\alpha_{p_1}^{-1}"]& p_{12}^\ast\Psi_2(p_1^\ast \widetilde{V}) \ar[r,"p_{12}^\ast \Psi_2(\varphi)"]  &p_{12}^\ast\Psi_2(p_2^\ast \widetilde{V}) \ar[r,"p_{12}^\ast \alpha_{p_2}"]  
	&\pi_2^\ast \Psi(\widetilde{V})
\end{tikzcd}\]
and analogously for $p_{12}$ replaced by $p_{23}$ or $p_{13}$. If this is the case, note that the bottom map is $p_{12}^\ast \psi$, so this implies that
\[p_{23}^\ast \psi \circ p_{12}^\ast \psi     
= \alpha_{\pi_3}\circ \Psi_3(p_{23}^\ast\varphi)\circ \alpha_{\pi_2}^{-1} \circ \alpha_{\pi_2}\circ \Psi_3(p_{12}^\ast\varphi)\circ \alpha_{\pi_1}^{-1}     
=\alpha_{\pi_3}\circ \Psi_3(p_{13}^\ast\varphi)\circ \alpha_{\pi_1}^{-1}     
= p_{13}^\ast \psi\]
as desired. To see the claim, we use that by \Cref{p: functorial}.(2), the relation $\pi_1=p_{1}\circ p_{12}$ implies that we have $\alpha_{\pi_1}=p_{12}^\ast\alpha_{p_{1}}\circ \alpha_{p_{12}}$. Hence the above diagram reduces to
\[\begin{tikzcd}
	\Psi_3(\pi_1^\ast\widetilde{V})\ar[r,"\Psi_3(p_{12}^\ast\varphi)"]  \ar[d,"\alpha_{p_{12}}"]& \Psi_3(\pi_2^\ast\widetilde{V})\ar[d, "\alpha_{p_{12}}"]\\
	p_{12}^\ast\Psi_2(p_1^\ast \widetilde{V}) \ar[r,"p_{12}^\ast \Psi_2(\varphi)"] & p_{12}^\ast\Psi_2(p_2^\ast \widetilde{V}) .
\end{tikzcd}\]
This commutes by \Cref{p: functorial}.(1) and hence we are done.

The same argument also works in the other direction for $\Psi^{-1}$. 
\end{proof}
\begin{Example}
Let $X_\infty\to X$ be an anti-canonical cover of an ordinary abeloid variety as in \Cref{l:abeloid-assumption}. While it is in general difficult to make the implicit constant $c$ in \Cref{d:smallness} and hence it \Cref{d:loc-small} more explicit, we can always find interesting examples of locally small $G$-bundles on $X_\mathrm{v}$, illustrating the power of \Cref{p: Global-small-correspondence}:
\begin{itemize}
	\item Any \'etale-locally free $G$-bundle $V$ on $X_\mathrm{v}$ is clearly locally small. Unravelling the construction, we have $\mathrm{S}_f^{\lsm}(V)=(V,0)$. This example is the reason why we use \'etale localisation in \Cref{d:loc-small} rather than analytic localisation.
	\item The proof of \Cref{l:abeloid-assumption} shows that one can always find an affinoid cover $\wt X\to X$ such that $\wt X$ and the double and triple fibre products $\wt X_2$ and $\wt X_3$ satisfy  \Cref{a:local-cover}. Taking $\gamma>0$ such that \Cref{Lemma2.14} holds for $\wt X$, $\wt X_2$ and $\wt X_3$, we see that for $c\gg 0$, any continuous representation $\rho:\Delta_X\to G_c(K)$ defines a locally small $G$-bundle $V_\rho$ on $X$ by pushout of $X_\infty\to X$ along $\rho$. Describing the underlying $G$-bundle of the associated Higgs bundle $S_f^{\lsm}(V_\rho)$ is a non-trivial task already for $G=\G_m$.
\end{itemize}
\end{Example}

\section{Pro-finite-\'etale torsors} 
In this section, we move on to Questions~\ref{q3} and \ref{q2}: We therefore study $G$-torsors for the v-topology which become trivial on the pro-finite-\'etale cover of \Cref{s:profet-univ-cover}:

Let $X$ be a connected smooth proper rigid space over $K$ with a fixed point  $x_0\in X(K)$, and denote by $\wt X \to X$ the diamantine pro-finite-\'etale universal cover as in \Cref{d:univ-cover}, which is a  $\pi:=\pi_1^\et(X,x_0)$-torsor. The main goal of this section is to prove:
\begin{Theorem}\label{rslt:vbundles}
Let $X$ be a connected smooth proper rigid space over $K$ with a fixed base-point ${x_0 } \in X(K)$. Let $G$ be a rigid group variety. We consider the natural functor of groupoids
\begin{alignat*}{3}
	\Big\{\begin{array}{@{}c@{}l}\text{continuous homomorphisms }\\
		\pi_1^\et(X,{x_0}) \rightarrow G(K) \end{array}\Big\}&\to\,\,&&  \Big\{\begin{array}{@{}c@{}l}\text{ pro-finite-\'etale}\\
		\text{$G$-torsors on $X_\mathrm{v}$}\end{array}\Big\}.\\
	\rho\quad &\mapsto &&\wt X \times^{\pi_1^\et(X,{x_0}),\rho}G
\end{alignat*}
sending a continuous homomorphisms  $\rho$ to the pushout of $\wt X\to X$ along $\rho$.
\begin{enumerate}
	\item If $G$ is connected or commutative, then this functor is essentially surjective. 
	\item It is fully faithful if there are no non-constant morphisms $\wt X\to G$. For example this happens if $G$ is linear analytic, i.e.\ admits an embedding into $\mathrm{GL}_n$ for some $n$.
\end{enumerate}
\end{Theorem}
Here as in \Cref{p:pirep}, the morphisms $\rho_1\to \rho_2$ between two continuous representations on the left hand side are defined as the set of elements $g\in G(K)$ such that
$g\cdot\rho_1(\gamma)\cdot g^{-1}=\rho_2(\gamma)$ for all $\gamma\in \pi_1^\et(X,{x_0})$.
In the case that $G=\GL_n$, this means that the morphisms are the isomorphisms between representations. 
\begin{Remark}
While we do not know if (1) holds for any $G$, the fully faithfulness in (2) can fail in general. We refer to \Cref{r:rmk-fully-faihful-new} below for a concrete example.
\end{Remark}
\begin{Remark}
Recall that $\pi_1^\et(X,{x_0})$ is defined in terms of connected finite \'etale covers of $X$. In rigid analytic geometry, there is also a second fundamental group, namely the de Jong fundamental group $\pi_1^{\mathrm{dJ}}(X,x_0)$, defined in terms of a different kind of \'etale covers which are not necessarily finite. While one can show that it is still possible to attach $G$-torsors on $X_v$ to continuous representations of $\pi_1^{\mathrm{dJ}}(X,x_0)$, this will typically not result in a fully faithful functor, because the analogue of the universal cover for $\pi_1^{\mathrm{dJ}}(X,x_0)$ has many global sections.
\end{Remark}
\begin{proof}[Proof of \Cref{rslt:vbundles}.(2)]
It is clear that the displayed functor is well-defined. If there are no non-constant maps $\wt X\to G$, this means that $G(\wt X)=G(K)$ and hence
\[H^1_{\cts}(\pi,G(\wt X))=H^1_{\cts}(\pi,G(K))=\Hom_{\cts}(\pi,G(K))\]
as the $\pi$-action on $G(K)$ is trivial. In the case $G$ is linear analytic, since by \cite[Proposition~3.9]{heuer-geometric-Simpson-Pic}  we have $\GL_n(\wt X)\subseteq M_n(\wt X)\cong \O(\wt X)^{n\times n}=K^{n\times n}$, every map $\wt X\to \GL_n$ is constant. It follows that $G(\wt X)=G(K)$ as well.
\end{proof}
\Cref{rslt:vbundles}.(1) is more difficult and will take us the rest of this section.
\subsection{Cocycles}%\subsection{Morphisms from universal covers into rigid groups}
\label{sec:morphisms}

In order to describe the $G$-torsors on $X$ that are trivialised by  $\wt X$, we will study the \v{C}ech nerve of the $\pi$-torsor $\wt X\to X$. As a first step, we now reduce this to a computation of group cohomology.
In doing so, to simplify notation, let us drop the underline in the notation of the diamond $\underline{\pi}$, i.e.\ we freely move back and forth between profinite sets and their associated diamonds. 

\begin{Lemma}
Let $G$ be any rigid group, and let $n\in \N$ if $G$ is commutative, and $n \in  \{0,1\}$ otherwise. Then there  is a natural isomorphism \[\check{H}^n(\wt X\to X,G)=H^n_{\cts}(\pi,G(\wt X)),\]
where $G(\wt X)$ is endowed with its natural topology via \Cref{p:can-grp-struct}.
\end{Lemma}
\begin{Example}\label{ex:geom-cocycle}
Let us clarify our conventions for non-abelian cocycles:  Since $\pi$ acts on $\wt X$ from the left, the induced action on $G(\wt X)$ is a right-action. For $n=1$, the 1-cocycles defining $H^1_\cts(\pi,G(\wt X))$ are therefore concretely the morphisms 
\[\rho:\pi\times \wt X\to G\]
for which the natural diagram expressing the relation
\begin{equation}\label{eq:non-abelian-Z1}
	\rho(\gamma_1\cdot \gamma_2,x)=\rho(\gamma_1,\gamma_2x)\cdot \rho(\gamma_2,x)
\end{equation}
commutes. The equivalence relation defining $\check{H}^1$ is then defined by declaring that two such cocycles $\rho_1$ and $\rho_2$ are equivalent if and only if there is a morphism $\varphi\in G(\wt X)$ such that 
\begin{equation}\label{eq:non-abelian-H1}\rho_1(\gamma, x)=\varphi(\gamma x)\rho_2(\gamma,x)\varphi(x)^{-1}.\end{equation}
\end{Example}
\begin{proof}
Since $\wt X\to X$ is a $\pi$-torsor, the $(n+1)$-fold fibre product of $\wt X$ with itself over $X$ is isomorphic to ${\pi}^n\times \wt X$.
Hence cocycles in the \v{C}ech cohomology set with $G$-coefficients $\check{H}^n(\wt X\to X,G)$  are given by morphisms of diamonds 
\[\rho:{\pi}^n\times \wt X\to G.\]
Via  \Cref{l:top-on-G(T)-test-by-Hausd-sp}, such morphisms $\rho$ correspond to continuous morphisms ${\pi}^n\to G(\wt X)$. By a standard computation, the \v{C}ech cocycle relation is identified for Galois covers like $\wt X\to X$ with the cocycle condition of continuous group cohomology, and the equivalence relations defining the cohomology on both sides agree via this identification.

Let us give some more details of this identification for $n=1$ in the non-commutative case, for which we do not know a reference in the literature: We use the isomorphisms
\begin{equation}\label{eq:proof-of-cocycle-ide-first-isom}
	\pi\times \wt X\isomarrow \wt X\times_X\wt X,\quad (\gamma,x)\mapsto (\gamma x,x)
\end{equation}
\[ \pi\times \pi\times \wt X\isomarrow \wt X\times_X\wt X\times_X\wt X,\quad (\gamma_1,\gamma_2,x)\mapsto (\gamma_1\gamma_2 x,\gamma_2x,x).\]
It is a straightforward calculation that this identifies the projections $p_{ij}:\wt X\times_X\wt X\times_X\wt X\to \wt X\times_X\wt X$ with the maps $p_{ij}:\pi\times \pi\times \wt X\to \pi\times \wt X$ defined as follows
\[ p_{12}(\gamma_1,\gamma_2,x)=(\gamma_1,\gamma_2x),\quad p_{23}(\gamma_1,\gamma_2,x)=(\gamma_2,x),\quad p_{13}(\gamma_1,\gamma_2,x)=(\gamma_1\gamma_2,x).\]
By definition, 1-cocycles in $\check{H}^1(\wt X\to X,G)$ are now given by sections $\rho\in G(\wt X\times_X\wt X)=G(\pi\times \wt X)=\mathrm{Map}_{\cts}(\pi,G(\wt X))$ satisfying $p_{12}^\ast(\rho)\cdot p_{23}^\ast(\rho)=p_{13}^\ast(\rho)$. Via the above identifications, this translates to
\[ \rho(\gamma_1\cdot \gamma_2,x)=\rho(\gamma_1,\gamma_2x)\rho(\gamma_2,x) \]
for all $\gamma_1,\gamma_2\in \pi$, $x\in \wt X$,
which is precisely the same as in \Cref{eq:non-abelian-Z1}. 

Finally, the equivalence relation defining $\check{H}^1(\wt X\to X,G)$ is given by saying that $\rho_1\sim\rho_2$ if there is $\varphi\in G(\wt X)$ such that $\rho_1=p_1^{\ast}(\varphi)\rho_2p_2^{\ast}(\varphi)^{-1}$ where $p_1,p_2:\wt X\times_X\wt X\to \wt X$ are the projections. When we identify these projections with maps $\pi\times \wt X\to \wt X$ via \eqref{eq:proof-of-cocycle-ide-first-isom}, this equivalence relation gets identified with \eqref{eq:non-abelian-H1}.
\end{proof}

Consider now the map $G(K)\to G(\wt X)$. Since the $\pi$-action on $G(K)$ is trivial, this induces a natural map
\[ \Homc(\pi,G(K))\to {H}^1_{\cts}(\pi,G(\wt X))\]
defined in terms of geometric 1-cocycles by sending 
\[\rho\mapsto \big(\pi\times \wt X\to G,\quad \gamma,x\mapsto \rho(\gamma)\big).\]

The key result that we need for the proof of \Cref{rslt:vbundles}.(1) is the following:
\begin{Theorem}\label{p:Hom-to-H^1-surj}
Let $G$ be any connected rigid group, $X$ any connected smooth proper rigid space. Then the following natural map is surjective:
\[ \Hom_{\cts}(\pi,G(K))\to H^1_{\cts}(\pi,G(\wt X)).\]
\end{Theorem}
We currently do not know if the connectedness assumption on $G$ is necessary. However, we will see in \Cref{t:H1cts-pi-G-for-commutative-G} that it can be removed in the commutative case.

For the proof of \Cref{p:Hom-to-H^1-surj}, we first need some preparations. We start by discussing a variant of the classical Rigidity Lemma from algebraic geometry for $\wt X$.

\subsection{A Rigidity Lemma for universal covers}
The classical Rigidity Lemma in algebraic geometry says that any morphism between  $K$-schemes $f:X\times Y\to Z$ where $X$ is irreducible and proper with a $K$-point $x_0$, $Y$ is irreducible with a $K$-point $y_0$ and $Z$ is any scheme, and such that $f(X\times \{y_0\})$ is a single point,  factors through the projection $X\times Y\to Y$. The analogous statement also holds in the rigid analytic setting \cite[Lemma~7.1.2]{Luetkebohmert}.

We need an analog of this statement for $\wt X$. This will be weaker because $\wt X$ is not in any sense ``irreducible'': For example, in the case of the Tate curve $X$, we easily verify from the explicit description of \cite[Theorem 4.6.3]{AWS-perfectoid} that any open neighbourhood of the identity in $\wt X$ contains an open subspace of the form $\Z_p\times \mathbb B^{1,\mathrm{perf}}$, which is manifestly not irreducible as it has many connected components. 

We therefore need to prove an analog of a slightly weaker version:
If the irreducibility assumptions in the algebraic Rigidity Lemma are dropped, one can still deduce that $f$ factors through the projection on $X\times U$ where $U\subseteq Y$ is an open neighborhood of $y_0$. 

\begin{Lemma}[Rigidity Lemma]\label{l:Rigidity}		Let $Y$ be a diamond and let $Z$ be any adic space over $K$. Assume that there is a $K$-point $y_0\in Y(K)$ and let
\[ f:\wt X\times Y\to Z\]
be any morphism of diamonds such that $f(\wt X\times y_0)=z_0:=f(x_0,y_0)$ is a single point. 

Then there is an open neighbourhood $y_0\in U\subseteq Y$ such that $f|_{\wt X\times U}$ factors through the projection to $U$ composed with a morphism of diamonds $U\to Z$.
\end{Lemma}
\begin{proof}
We first observe that the statement is quasi-pro-\'etale-local on $Y$ (note that we may lift the base-point $y_0\in Y(K)$ to quasi-pro-\'etale covers since $K$ is algebraically closed), so we may replace $Y$ by an affinoid perfectoid space.

Let now $z_0\in V\subseteq Z$ be an affinoid open neighbourhood of $z_0=f(x_0,y_0)$.
Passing to complements, let $W:=(\wt X\times Y)\backslash f^{-1}(V)$. Since $\wt X\to X$ is a proper morphism of diamonds in the sense of \cite[ Definition 18.1]{etale-cohomology-of-diamonds}, the morphism $\wt X\to \Spa(K)$ is proper, thus so is the projection $q:\wt X\times Y\to Y$. It follows that $q(W)\subseteq Y$ is closed. Let $U$ be the complement. By assumption, we have $y_0\in U$, and it follows from the definition that $f$ restricts to a map
\[ \widetilde{X}\times U\to V.\]
Since $V$ is affinoid, this is determined by its global sections. But by \cite[Proposition~3.9]{heuer-geometric-Simpson-Pic},
$\O(\wt X\times U)=\O(U)$.
Hence this restriction factors through the projection to $U$.
\end{proof}

\subsection{Morphism from proper rigid spaces to rigid groups}
The goal 
of this section is to prove the following structural result for rigid groups: 
%\begin{Theorem}\label{t:morph-proper-to-grp}
%     Let $G$ be a commutative rigid group variety. 
%	Let $\varphi:X\to G$ be a morphism from any connected smooth proper rigid space $X$ to any commutative rigid group variety $G$. Assume that $0\in \im(\varphi)$. Then $\varphi$ factors through a closed subgroup $A\subseteq G$ which is an abeloid variety.
%\end{Theorem}

\begin{Theorem}\label{t:morph-proper-to-grp}
Any rigid group variety $G$ admits a maximal abeloid closed  subgroup $A\subseteq G$. This satisfies the following universal property:
Any morphism $\varphi:X\to G$ from an irreducible proper rigid space $X$ with $0\in \im(\varphi)$ factors through $A\subseteq G$. Moreover, $A\subseteq Z(G^\circ)$.
\end{Theorem}

The most important step in the proof is the following Lemma. 

\begin{Lemma}\label{cl:maximal-abeloid}
Any morphism $\varphi:X\to G$ from an irreducible proper rigid space $X$ with $0\in \im(\varphi)$  factors through \textit{some} closed abeloid subgroup $A\subseteq G$.
\end{Lemma}

The proof  relies crucially on the notion of ``analytic subsets'' of a rigid space $Y$ as introduced in \cite[\S9.5]{BGR}, that we now recall: To avoid confusion with the ``analytic topology'', we are instead going to call these ``Zariski-closed'' subspaces, namely $Z\subseteq Y$ is called Zariski-closed if for any affinoid open $U\subseteq Y$, the subspace $Z\cap U\subseteq U$ is the zero-locus of a set of elements $f_1,\dots,f_d\in \O(U)$. Equivalently, $Z$ is the image of a closed immersion defined by a coherent sheaf of ideals in $\O_Y$ \cite[\S9.5 Corollary~7]{BGR}. Given a Zariski-closed subspace $Z\subseteq Y$, it is always possible to endow it with the induced-reduced structure by replacing any coherent ideal defining $Z$ with its radical \cite[\S9.5.3, Proposition~4]{BGR}.

The notion of Zariski-closed subspaces has good localisation properties, i.e.\ whether a subspace is Zariski-closed can be checked analytic-locally, and the pullback of a Zariski-closed subset under a morphism $X\to Y$ of rigid spaces is again Zariski-closed \cite[\S9.6.3, before Lemma~4]{BGR}. 
Moreover, we have the following result due to Kiehl:
\begin{Proposition}[{\cite[\S9.6.3, Proposition~3]{BGR}}]\label{p:Kiehl-proper-mapping}
Let $f:X\to Y$ be a proper morphism of rigid spaces. Then the image $f(M)$ of any Zariski-closed subset $M\subseteq X$ is again Zariski-closed.
\end{Proposition}
\begin{Proposition}\label{p:im-is-proper}
Let $X$, $Y$ be rigid spaces over $K$. Assume that $X$ is proper and $Y$ is separated over $K$, and let $f:X\to Y$ be a morphism of rigid spaces. Then $f$ is proper and $f(X)$ endowed with its induced reduced structure is a proper rigid space. 
\end{Proposition}
\begin{proof}
Since $X$ is quasi-compact, we can without loss of generality assume that $Y$ is quasi-compact. Then by \cite[Corollary~5.10.(a)]{BLR-II}, $f$ has a formal model $\mathfrak f:\mathfrak X\to \mathfrak Y$. Since $Y$ is separated, $\mathfrak Y$ is separated (see e.g.\ \cite[Remark 1.3.18.(i) and Remark 1.3.19.(ii)]{huber-Etale-Adic}). We now use that by \cite[Corollaries~4.4, 4.5]{Temkin_local-properties}, the following are equivalent for a morphism $g$ of formal schemes:
\begin{enumerate}
\item $g$ is proper,
\item the rigid generic fibre of $g$ is proper,
\item the special fibre $g_0$ of $g$ over the residue field $k$ is proper.
\end{enumerate}
This first shows that $\mathfrak X$ is proper. Arguing mod $p^n$ for every $n$, it now follows from \cite[01W6]{Stacks} that $\mathfrak f$ is proper, hence in a second step we deduce that $f$ is proper.

To see that $f(X)$ is proper, we may replace $Y$ with $f(X)$ and assume that $f$ is surjective. We can now use the functorial reduction map $\mathrm{red}:Y(K)\to \mathfrak Y(k)$ induced by $\mathfrak Y$, which is surjective by \cite[\S 7.1.5 Theorem 4]{BGR}, to see that surjectivity of $f$ implies surjectivity of the special fibre $\mathfrak f_0:\mathfrak X_0\to \mathfrak Y_0$. Then by \cite[03GN]{Stacks}, $\mathfrak Y_0$ is proper, hence $Y$ is proper.
\end{proof}

As studied in \cite{Conrad_irreducible_components_of_rigid_spaces}, there is a good notion of irreducible subspaces of rigid spaces: A non-empty subspace $Z\subseteq X$ of a rigid spaces is called irreducible if it cannot be written as the union $Z=V_1\cup V_2$ of two Zariski closed subspaces $V_{1},V_2\subsetneq Z$, see \cite[Lemma~2.2.3]{Conrad_irreducible_components_of_rigid_spaces}. One checks as usual (c.f. {\cite[0379]{Stacks}}):
\begin{Lemma}\label{l:image-of-reduced}
	\begin{enumerate}
\item Let $X$ and $Y$ be irreducible rigid spaces over the algebraically closed field $K$. Then the fibre product $X\times Y$ is irreducible.
\item Let $f:X\to Y$ be any morphism of rigid spaces. Then the image of any irreducible subspace is irreducible.
\end{enumerate}
\end{Lemma}
\begin{proof}
	Part (2) follows from the fact that pullbacks of Zariski-closed sets are Zariski-closed.
	
	Part (1) can be deduced from \cite[Lemma~2.2.3]{Conrad_irreducible_components_of_rigid_spaces}, which says that $X$ is irreducible if and only if the Zariski-closure of any non-empty open subspace is all of $X$: Let $U\subseteq X\times Y$ be any non-empty open subspace. Since $K$ is algebraically closed, this contains a $K$-point $(x,y)\in X(K)\times Y(K)$. It follows that $U$ contains an open of the form $V\times W$ where $V\subseteq X$ and $W\subseteq Y$ are non-empty opens. Since $\overline{V}=X$ and $\overline{W}=Y$, it now suffices to see that the Zariski-closure of $V\times W$ is $\overline{V}\times \overline{W}$. But this can be seen as in algebraic geometry: Any local function $f$ that vanishes on $v\times W$ for each $v\in V(K)$ vanishes on $v\times \overline{W}$, hence on $V\times \overline{W}$. Then it also vanishes on $V\times w$ for each $w\in \overline{W}(K)$, hence on $\overline{V}\times w$, hence on $\overline{V}\times \overline{W}$.
\end{proof}

Let now $X$ be any irreducible proper rigid space, $G$ a commutative rigid group, and let us consider a morphism $\varphi: X \to G$ containing $0$
in the image. Then for any $n\in \N$, we define a morphism 
\[\textstyle \varphi_n:X^n\to G,\quad (x_1,\dots,x_n)\mapsto \sum_{i=1}^n(-1)^i\varphi(x_i).\]
Let $Z_n\subseteq G$ be the image of $\varphi_n$. The fibre product $X^n$ is still proper, and since $K$ is algebraically closed, it is irreducible by \Cref{l:image-of-reduced}.2. Since $G$ is separated, $\varphi_n$ is proper by \Cref{p:im-is-proper}, hence $Z_n$ is Zariski-closed by \Cref{p:Kiehl-proper-mapping}. We can endow it with its induced-reduced structure, so $Z_n$ is a reduced rigid space. The key technical result is now:
\begin{Lemma} \label{l:ascending-chain-stablize}
Let  $\varphi: X \to G$ be a morphism  from an irreducible proper rigid space to a commutative rigid group containing  $0$
in the image, and let $Z_n$ be as above. 

Then each $Z_n$ is proper.
The $Z_n$'s form an ascending chain
\[ Z_0\subseteq Z_1\subseteq Z_2\subseteq...\]
of equidimensional irreducible Zariski-closed subspaces of $G$. This chain stabilises.
\end{Lemma}
\begin{proof}
That the $Z_n$ is proper follows from \Cref{p:im-is-proper}.
Each $X^n$ is irreducible by \Cref{l:image-of-reduced}.2, and hence $Z_n$ is irreducible by \Cref{l:image-of-reduced}.1. That the $Z_n$'s form an ascending chain follows from the assumption that $0\in \im(\varphi)$: Let $e\in X(K)$ be any preimage of $0$ under $\varphi$, then the diagram
\[\begin{tikzcd}
X^n \arrow[d, "\mathrm{id} \times \{e\}"] \arrow[r, "\varphi_n"] & G \arrow[d, equal] \\
X^{n+1} \arrow[r, "\varphi_{n+1}"]                       & G
\end{tikzcd}\]
commutes, and thus $Z_n\subseteq Z_{n+1}$.

By \cite[paragraph before Lemma~2.2.3]{Conrad_irreducible_components_of_rigid_spaces}, any irreducible rigid space is equidimensional, meaning that for fixed $n$, the Krull dimension $\dim \mathcal O_{Z_{n},z}$ of the stalks is constant for all $z\in Z_n(K)$. Since $G$ is a rigid space, we know that $\dim \mathcal O_{Z_{n},z}\leq \dim \O_{G,z}<\infty$, from which it follows that there is $n\in \N$ such that the dimension of $Z_m$ is the same for all $m\geq n$. Then $Z_n\subseteq Z_m$ is an analytic subspace of the irreducible rigid space $Z_m$ which is equidimensional of the same dimension as $Z_m$. By \cite[Corollary 2.2.7]{Conrad_irreducible_components_of_rigid_spaces}, it follows that $Z_n=Z_m$.  %This can be checked locally on $G$, so we may assume for the moment that $G=\Spa(A)$ is affinoid and all $Z_m\subseteq G$ are affinoid. Then the $Z_m$'s, being irreducible and reduced, correspond to a descending chain of prime ideals $\mathfrak p_m$ of $A$, and $\dim Z_m=\dim(A/\mathfrak p_m)$. Since $G$ is smooth, $A$ is regular \cite[\S1.2]{Conrad_irreducible_components_of_rigid_spaces} hence $A$ is catenary. It thus has a dimension function in the sense of \cite[0ECF]{Stacks}, which means that for $\mathfrak p_m\subseteq \mathfrak p_n$, the fact that $\dim(A/\mathfrak p_n)=\dim(A/\mathfrak p_m)$ implies that $\mathfrak p_n=\mathfrak p_m$. Hence $Z_n=Z_m$.
\end{proof}

\begin{proof}[Proof of \Cref{cl:maximal-abeloid}:]
Since $X$ is irreducible, it is in particular connected, so we can without loss of generality assume that $G$ is connected by replacing $G$ with $G^\circ$.
We can further reduce to the case that $G$ is commutative by the rigid version of a standard argument in classical algebraic geometry: We consider the morphism of rigid spaces
\[ f:X\times G\to G, \quad x,g\mapsto \varphi(x)g\varphi(x)^{-1}g^{-1}\]
This satisfies $f(x,1)=1$ for all $x \in X$. Hence the rigid analytic version of the Rigidity Lemma \cite[Lemma~7.1.2]{Luetkebohmert} gives an open neighbourhood $1\in U\subseteq G$ on which $f|_{X\times U}$ factors through $U$. Since $f(x_0, g) = 1 $ for all $g \in G$ by assumption, we find that $f|_{X\times U}=1$.
It follows that for any $K$-point  $z\in X(K)$, the induced morphism of rigid spaces \[\psi:G\to G,\quad g\mapsto \varphi(z)g\varphi(z)^{-1}g^{-1}\] 
restricts to $\psi_{|U}=1$ on the non-empty open subspace $U$. As $G$ is irreducible, $\psi$ is uniquely determined by $\psi|_U$. This shows that $\psi=1$ on all of $G$. Thus $\varphi$ factors through $Z(G)$.

Hence we may assume that $\varphi: X \rightarrow G$ is a morphism to a commutative rigid group containing $0$ in the image. By \Cref{l:ascending-chain-stablize},  it follows from this that there is a proper rigid space $Z=\cup_{m\in\N} Z_m=Z_n$ for $n\gg 0$ with the property that $\varphi$ factors through a closed immersion $X\to Z\to G$. Moreover, by construction, $Z$ contains $0$ and the diagram
\[
\begin{tikzcd}
Z\times Z \arrow[r] \arrow[d] & Z \arrow[d] \\
G\times G \arrow[r, "m"]     & G   
\end{tikzcd}\]
commutes,
where the top line is the map induced by the natural morphism $Z_{2n}\times Z_{2n}\to Z_{4n}$.

Due to the alternating sign, $Z$ is also closed under inverses: Namely, the morphisms $Z_n\to Z_{n+1}$ induced by the map $X^n\to X^{n+1}$, $(x_1,\dots,x_n)\mapsto (e,x_1,\dots,x_n)$ induce a morphism $Z\to Z$ over $[-1]:G\to G$.

All in all, it follows that $Z\subseteq G$ is a rigid subgroup of $G$ which is proper and irreducible, in particular connected. Hence it is an abeloid subvariety of $G$.
\end{proof}

\begin{proof}[Proof of \Cref{t:morph-proper-to-grp}]

Let $A,B\subseteq G$ be any two closed abeloid subgroups. Then $A\times B$ is still smooth proper and connected. By \Cref{cl:maximal-abeloid} we find that the map
\[ A\times B\to G\times G\xrightarrow{m} G,\]
factors through
an abeloid subvariety $C\subseteq G$ that contains $A$ and $B$. It follows that the union of all abeloid subvarieties of $G$ is an abeloid subvariety. Now the Theorem follows from \Cref{cl:maximal-abeloid}.
\end{proof}    
\subsection{Cocycles from homomorphisms: The commutative case}
With these preparations, we can now prove \Cref{p:Hom-to-H^1-surj} for commutative $G$ (not necessarily connected), or more precisely:
\begin{Theorem}\label{t:H1cts-pi-G-for-commutative-G}
Let $G$ be any commutative rigid group, $X$ any connected smooth proper rigid space. Then the following natural map is surjective:
\[ \Hom_{\cts}(\pi,G(K))\to H^1_{\cts}(\pi,G(\wt X)).\]
\end{Theorem}
\begin{proof}
Consider the evaluation morphism 
\[ G(\wt X)\xrightarrow{ev_{x_0}}G(K)\]
at the fixed base point $x_0\in \wt X$.
This is a surjective group homomorphism, which is split by sending $g\in G(K)$ to the constant map $\wt X\to \Spa(K)\xrightarrow{g}G(K)$. Denote by $G(\wt X)^\circ$ the kernel of this homomorphism $ev_{x_0}$, this is the subgroup of those morphism $\wt X\to G$ that send $x_0$ to $0$. Since $G$ is commutative, the splitting of $\mathrm{ev}_{x_0}$ now induces a short exact sequence
\begin{equation}\label{eq:ses-G(wtX)^0}
0\to G(K)\to G(\wt X)\xrightarrow{\phi} G(\wt X)^\circ\to 0,
\end{equation}
where  $\phi$ sends $\varphi:\wt X\to G$ to the map $\varphi^\circ$ making the following diagram commute:
\[
\begin{tikzcd}
\wt X \arrow[r,"\varphi"]                                      & G                                \\
\wt X \arrow[r, "\varphi^\circ"] \arrow[u, equal] & G. \arrow[u, "\cdot \varphi(x_0)"'.]
\end{tikzcd}\]
We note that $\phi$ is a group homomorphism since $G$ is commutative.
Observe now that the first morphism $G(K)\to G(\wt X)$ in the exact sequence is clearly \mbox{$\pi$-equivariant}. Consequently, $G(\wt X)^\circ$ inherits a natural quotient action (note that the action of  $\pi$ does not preserve  $G(\wt X)^\circ$ as a subspace of $G(\wt X)$).
Explicitly, the quotient action of $\pi$ on $G(\wt X)^\circ$ is defined by letting $\gamma\in \pi$ send $\varphi$ to the morphism $\varphi^\gamma$ defined by making the following diagram commutative:
\begin{equation}\label{eq:quotient-action-on-G(wtX)^circ}
\begin{tikzcd}
	\wt X \arrow[r, "\varphi"]                                      & G                                \\
	\wt X \arrow[r, "\varphi^\gamma"] \arrow[u, "\cdot \gamma"] & G. \arrow[u, "\cdot \varphi(\gamma x_0)"']
\end{tikzcd}
\end{equation}
Recall now from \Cref{p:can-grp-struct} that $G(\wt X)$ has a natural topology defined by a system of open neighbourhoods $G_k(\wt X)$ where $(G_k)_{k\in \N}$ is a system of subgroups of $G$. When we endow $G(\wt X)^\circ\subseteq G(\wt X)$ with the natural subspace topology, then the section $\phi:G(\wt X)\to G(\wt X)^\circ$ is continuous: Indeed, it is immediate from the definition that $\phi^{-1}(G_k(\wt X)^\circ)\supseteq G_k(\wt X)$.

In summary, this shows that \eqref{eq:ses-G(wtX)^0} is a short exact sequence of topological $\pi$-modules. We deduce that there is an  exact sequence
\begin{equation}\label{eqn:cohomology}
\Hom_{\cts}(\pi,G(K))\to H^1_{\cts}(\pi,G(\wt X))\to H^1_{\cts}(\pi,G(\wt X)^\circ).
\end{equation}

More explicitly, this can be interpreted as follows: any geometric $1$-cocycle $\pi\times \wt X\to G$ can be written as a product of a continuous map $\pi\to G(K)$ obtained by specialising at $x_0$ (not necessarily a homomorphism) and a morphism $\rho: \pi\times \wt X\to G$ such that $\rho(\gamma,x_0)=0$. 

We wish to see that the last term in the exact sequence vanishes.
The crucial point is:
\begin{Lemma}\label{l:G(wtX)^0-discrete-top}
The subspace topology of $G(\wt X)^\circ\subseteq G(\wt X)$ is the discrete one.
\end{Lemma}
\begin{proof}
Let $G_k\subseteq G$ be any affinoid open subgroup isomorphic to $\mathbb B^d$ as in \Cref{p:can-grp-struct}, then the topology of $G(\wt X)$ is defined by decreeing that $G_k(\wt X)$ is open. But $G_k(\wt X)\cong \mathbb B^d(\wt X)=\O_K^d$. Hence $G_k(\wt X)\cap G(\wt X)^\circ=\{0\}$ is open, and  $G(\wt X)^\circ$ is discrete.
\end{proof}

\begin{Lemma}\label{l:pi-action-on-G(wt X)^circ}
Let $\varphi:\wt X\to G$ be a morphism in $G(\wt X)^\circ$, then there is an open subgroup $U\subseteq \pi$ such that the action of $U$ as defined in \Cref{eq:quotient-action-on-G(wtX)^circ} fixes $\varphi$.
\end{Lemma}
\begin{proof}
We apply the Rigidity Lemma \ref{l:Rigidity} to the morphism
\[ \wt X\times \underline{\pi}\to G,\quad (y,\gamma) \mapsto \varphi(\gamma x_0) \varphi(\gamma y)^{-1} \varphi(y)\]
and find an open subgroup $U$ of $\pi$ such that the restriction to $\wt X \times {U}$ vanishes.
By the diagram defining the action on $G(\wt X)^\circ$, this implies that $\varphi^\gamma=\varphi$ for all $\gamma \in U$.
\end{proof}

\begin{Lemma}\label{thm:factorization}
For any continuous $1$-cocycle $\rho:\pi\to G(\wt X)^\circ$,  there is a  connected finite \'etale cover $X'\to X$ with Galois group $Q_{X'}$ such that $\rho$ factors through $X'\times Q_{X'}$. Hence the following map is surjective
\[H^1_\cts(\pi,\varinjlim_{X'\to X}G(X')^\circ)\to H^1_\cts(\pi,G(\wt X)^\circ),\]
where  $G(X')^\circ:=\ker(G(X')\xrightarrow{\mathrm{ev}_{x_0}} G(K))$ is endowed with the discrete topology.
\end{Lemma}
\begin{proof}
By \Cref{l:G(wtX)^0-discrete-top}, $\rho$ is continuous from a profinite set to a discrete set.
It follows that it has finite image. Using \Cref{l:pi-action-on-G(wt X)^circ}, this shows that $\rho$ factors through $(G(\wt X)^\circ)^U$ for some open subgroup $U\subseteq \pi$. Shrinking $U$ further, we may assume that it moreover factors through $\pi/U$. This shows that
\[\varinjlim_{X'\to X} H^1(Q_{X'},G(X')^\circ)\to H^1_\cts(\pi,G(\wt X)^\circ).\]
is surjective. By \cite[Proposition 1.2.5]{NeuSchWin}, the left hand side equals $H^1_\cts(\pi,\varinjlim_{X'\to X}G(X')^\circ)$.
\end{proof}

\begin{Proposition}\label{rslt:vanishingcoh}
Let $G$ be any commutative rigid group and $X$ any connected smooth proper rigid space. Then
${H}^1_\cts(\pi,G(\wt X)^\circ)=0$.
\end{Proposition}
\begin{proof}
By \Cref{thm:factorization}, it suffices to prove that
$H^1(\pi,\varinjlim_{X'\to X}  G(X')^\circ)=0$.
By \Cref{t:morph-proper-to-grp}, there is a maximal abeloid subvariety $A\subseteq G$, and we have $G(X')^\circ=A(X')^\circ$. Consider now the short exact sequence of sheaves on $X_{\et}$
\[ 0\to A[N]\to A\xrightarrow{[N]} A\to 0.\]
Taking cohomology over $X'\to X$, this gives a long exact sequence
\[ 0\to A[N](X')\to A(X')\to A(X')\to H^1_\et(X',A[N]).\]
The first term is equal to $A[N](K)$ since $X'$ is connected. Using that $[N]:A(K)\to A(K)$ is surjective, it follows that we obtain a left-exact sequence
\[0\to A(X')^\circ\xrightarrow{[N]} A(X')^\circ\to H^1_\et(X',A[N]).\]
Taking the colimit over $X'\to X$, the last term vanishes since every class in $H^1(X',A[N])$ defines a finite \'etale torsor over $X'$ and is thus killed by some finite \'etale cover of $X$. It follows that
\[ [N]:\varinjlim_{X'\to X}A(X')^\circ\isomarrow \varinjlim_{X'\to X}A(X')^\circ \]
is an isomorphism. In particular, $\varinjlim_{X'\to X}A(X')^\circ$ is a $\Q$-vector space. By
\cite[Proposition 1.6.2.c]{NeuSchWin}, it follows that ${H}^1_{\cts}(\pi,\varinjlim_{X'\to X}A(X')^\circ)=0$, as desired. 
\end{proof}

\Cref{t:H1cts-pi-G-for-commutative-G} now follows from the exact sequence \eqref{eqn:cohomology} and \Cref{rslt:vanishingcoh}.
\end{proof}

\subsection{The case of general $G$}
Building on the commutative case, we can now deduce the case of connected rigid groups $G$: 

\begin{proof}[Proof of \Cref{p:Hom-to-H^1-surj}]
The key idea is to use the exact sequence from \Cref{l:center} 
\[0\rightarrow Z(G)\rightarrow G\xrightarrow{\mathrm{ad}} \mathrm{Aut}(\mathfrak{g}),\]
where $Z(G)$ is the center as in  \Cref{d:center}.

Working in the category of pro-\'etale sheaves on $\Perf_K$,
we denote by $H$ the sheaf theoretic image of $G$ in $\mathrm{Aut}(\mathfrak{g})$. Then we have a short exact sequence of sheaf of groups on $\Perf_{K,\mathrm{pro}\et}$
\[0\rightarrow Z(G)\rightarrow G\rightarrow H\rightarrow 0.\]
Evaluating this over $\wt{X}$ and over $K$, we get a commutative diagram of $\pi$-equivariant maps
\[
\begin{tikzcd}
0\ar[r] &Z(G)(\wt{X})\ar[r] & G(\wt{X})\ar[r] & H(\wt{X})\ar[r] &0\\
0\ar[r]&Z(G)(K)\ar[r] \ar[u]& G(K)\ar[r] \ar[u]& H(K) \ar[u,"\sim"]\ar[r] &0.  
\end{tikzcd}
\]
The map $H(K)\rightarrow H(\wt{X})$ in the right column is an isomorphism: This follows from the fact that $H\subseteq \mathrm{Aut}(\mathfrak g)\cong \mathrm{GL}_n$ for $n=\dim G$, and $\mathrm{GL}_n(\wt X)=\mathrm{GL}_n(K)$.

The rows of this diagram are still exact: the bottom row is exact by surjectivity of $G\rightarrow H$ and the fact that $\Spa(K)$ is strictly totally disconnected. The top row is right exact by surjectivity in the bottom row.
We now take continuous group cohomology of $\pi$ with coefficients in the above diagram. Since $Z(G)(\wt{X})\subseteq G(\wt{X})$ is a commutative normal subgroup, by \cite[Chapitre I, Proposition 43]{Serre-cohom-galoisienne}, we get a commutative diagram of pointed sets
\[
\begin{tikzcd}[column sep = 0.4cm]
0\to H^1_\cts(\pi, Z(G)(\wt{X}))\ar[r] & H^1_\cts(\pi, G(\wt{X}))\ar[r] & H^1_\cts(\pi, H(\wt{X}))\ar[r] &H^2_\cts(\pi,Z(G)(\wt{X}))\\
0\to H^1_\cts(\pi,Z(G)(K))\ar[r] \ar[u]& H^1_\cts(\pi, G(K))\ar[r] \ar[u]& H^1_\cts(\pi, H(K)) \ar[u,"\sim"]\ar[r] &H^2_\cts(\pi,Z(G)(K)) \ar[u].  
\end{tikzcd}
\]

The leftmost horizontal arrows have trivial kernel since $G(K)\to H(K)$ is surjective and thus so is $G(\widetilde X)^\pi\to H(\widetilde{X})^\pi=H(K)^\pi=H(K)$.  

By \Cref{t:H1cts-pi-G-for-commutative-G} the leftmost vertical arrow is surjective. We also have injectivity of the rightmost vertical arrow: Using the sequence (\ref{eq:ses-G(wtX)^0}), its kernel identifies with $H^1_\cts(\pi,Z(G)^\circ)$, which is trivial by \Cref{rslt:vanishingcoh}.

Now the crucial point is that $Z(G)(\wt{X})\subseteq G(\wt{X})$ is even a central subgroup.
By \cite[Chapitre I, Proposition 42]{Serre-cohom-galoisienne}, it follows that the maps
\[ H^1_\cts(\pi, G(K))\to  H^1_\cts(\pi, H(K))\quad \text{and} \quad H^1_\cts(\pi, G(\wt{X}))\to  H^1_\cts(\pi, H(\wt{X}))\]
are in fact torsors under the group $H^1_\cts(\pi, Z(G)(K))$.
We can therefore argue as in the proof of the 5-Lemma to see that the second vertical arrow from the left is surjective as well.
\end{proof}

Combining \Cref{p:Hom-to-H^1-surj} and \Cref{t:H1cts-pi-G-for-commutative-G} completes the proof of \Cref{rslt:vbundles}. \qed

\subsection{A counter-example}
It is not in general true that the functor from \Cref{rslt:vbundles} is fully faithful. For example, this fails if the natural map $\Hom_{\cts}(\pi,G(K))\to H^1_{\cts}(\pi,G(\wt X)) $ (which is surjective by  \Cref{p:Hom-to-H^1-surj}) is not injective. We now give an example where this happens: Let us assume for simplicity that $G$ is commutative.
Then by the long exact sequence of group cohomology for \Cref{eq:ses-G(wtX)^0}, the kernel of this map can be identified with the cokernel of
\[G(X)\to H^0(\pi,G(\wt X)^\circ).\]
Unravelling the explicit description of the $\pi$-action on $G(\wt X)$ in  \eqref{eq:quotient-action-on-G(wtX)^circ}, we see that the second term is the set of $g\in G(\wt X)^\circ$ for which $\gamma^\ast g-g\in G(K)$ for any $\gamma \in \pi$. This can be non-trivial, even in very good cases, as the following example demonstrates:
\begin{Proposition}\label{r:rmk-fully-faihful-new}
Let both $X=A$ and $G=B$ be abelian varieties, considered as rigid analytic spaces. Assume that $\Hom(A,B)=0$, but that there exists an isomorphism $\wt A\isomarrow \wt B$. Then the composition
\[\rho: \pi_1^\et(A,0)\to \wt A\isomarrow \wt B\to B=G\]
defines a non-trivial element in  $\ker\big(\Homc(\pi,G(K))\to H^1_{\cts}(\pi,G(\wt X))\big)$. In particular, the functor of \Cref{rslt:vbundles} is not fully faithful in this case.
\end{Proposition}
There are plenty of examples of abelian varieties $A$ and $B$ satisfying the assumptions of \Cref{r:rmk-fully-faihful-new} by the main result of \cite{heuer-isoclasses}: For example, we can take $A$ and $B$ to be any non-isogeneous simple abelian varieties of good reduction whose special fibres are isogeneous.
\begin{proof}
We first note that $\rho=0$ would imply that $\wt A\to \wt B$ sends $\pi_1^\et(A,0)$ to $\pi_1^\et(B,0)$ and hence would induce an isogeny $A\to B$, contradicting the assumptions. Hence $\rho\neq 0$.

On the other hand, by \cite[Theorem~1.9]{heuer-isoclasses}, we have 
\[G(\wt X)^\circ=\Hom(\wt A,B)=\Hom(\wt A,\wt B),\] where the last equality holds by \cite[Proposition 3.8, Corollary~3.10]{heuer-isoclasses}. Observe now that by definition, the $\pi$-action on $G(\wt X)^\circ$ as described in \eqref{eq:quotient-action-on-G(wtX)^circ} is such that $\pi$ acts trivially on homomorphisms. It follows that
\[H^0(\pi,G(\wt X)^\circ)=\Hom(\wt A,\wt B)\neq 0\]
which is non-trivial by assumption. All in all, this shows that $\rho$ comes from $G(\wt X)^\circ$ via the boundary map. Thus  the kernel of $\Hom_{\cts}(\pi,G(K))\to H^1_{\cts}(\pi,G(\wt X))$ does not vanish.
\end{proof}

\begin{Remark}
Regarding the setting of
\Cref{r:rmk-fully-faihful-new}, we note that already $G$-torsors on $X_\an$ are an interesting category when $G$ and $X$ are both abelian varieties:
For abelian varieties $A$ and $B$ over $K$ considered as rigid spaces, the Raynaud uniformization 
\[0 \rightarrow \Lambda \rightarrow E \rightarrow A \rightarrow 0\] associated to $A$  induces a homomorphism $\Hom(\Lambda, B) \rightarrow \text{Ext}^1(A,B) \rightarrow H^1(A,B)$.
If $A$ has totally degenerate reduction, this map has non-trivial image. Hence, in this case there exist analytic $B$-torsors over $A$ which are not algebraic, since 
every Zariski-$B$-torsor  over $A$ is trivial by \cite[\S 4, Lemme 4]{Serre-fibres}. It seems an interesting question if there exists a  classification of all analytic/\'etale torsors in this setting which is analogous to the complex analytic characterization of torus bundles over tori \cite{PalaisStewart}, and how this changes for the v-topology.
\end{Remark}

\section{The $p$-adic Lie algebra exponential for commutative rigid groups}
\subsection{Recollections on the $p$-adic Lie group logarithm}
Let $L$ be any non-archimedean field extension of $\Q_p$. We will later take $L=K$ algebraically closed, but for now we can work more generally.
Recall that in $p$-adic arithmetic, the $p$-adic logarithm defines a left-exact sequence of topological groups
\[ 0\to \mu_{p^\infty}(L)\to 1+\m_L \xrightarrow{\log} L.\]
This can be upgraded to a short exact sequence of rigid group varieties in the \'etale topology \cite[\S7]{dJ_etalefundamentalgroups}\cite[Lemma~2.18]{heuer-Line-Bundle} 
\[ 0\to \mu_{p^\infty}\to \mathbb{D}(1)\to \G_a\to 0,\]
where for any $a \in L$, we denote by $\mathbb{D}(a)$  the open rigid unit disc around $a$. There is a partial splitting of $\log$ given by the $p$-adic exponential
\[\textstyle\exp:p^{\alpha}\m_L\to 1+\m_L,\quad x\mapsto \sum_{n=0}^\infty \frac{x^n}{n!}\]
where $\alpha=\frac{1}{p-1}$ for $p>2$ and $\alpha=2$ for $p=2$. This can again be upgraded to a  homomorphism of rigid groups 
\[ \exp: p^{\alpha}\mathbb{D}(0)\to \mathbb{D}(1).\]

But in contrast to the complex case, the exponential series does not converge on all of $L$. Indeed, the logarithm sequence of rigid groups is never split as any morphism from $\G_a$  to $\mathbb{D}(1)$ is constant. Instead, one can only obtain non-canonical splittings on $L$-points:
\begin{Definition}\label{d:exp}
By an exponential map for $L$ we mean a continuous group homomorphism 
\[\exp:L\to 1+\m_L\] which splits the $p$-adic logarithm $\log:1+\m_L\to L$ and is a continuation of the canonical $p$-adic exponential on the open subgroup $p^{\alpha}\m_L\subseteq L$. 
\end{Definition}
\begin{Lemma}
An exponential map $\exp:L\to 1+\m_L$ exists if $H^1_{\et}(\Spa(L),\mu_{p})=1$. In particular, it exists if $L$ is algebraically closed.
\end{Lemma}
\begin{proof}
It suffices to find a group-theoretic splitting of $\log$ extending $\exp$, this will automatically be continuous if it extends $\exp$ on $p^{\alpha}\m_L$. To find this splitting, we consider the long exact sequence
\[ \dots \to \Hom(L,1+\m_L)\to \Hom(p^\alpha\m_L,1+\m_L)\to \mathrm{Ext}^1(L/p^\alpha\m_L,1+\m_L)\]
We claim that the last term vanishes: If $L$ is algebraically closed, this is because $1+\m_L$ is divisible, in particular injective (cf \cite[\S A.2.1]{xu2022parallel}). In general, since $L/p^\alpha\m_L$ is a $\Z_p$-module, it suffices to prove that $1+\m_L$ is $p$-divisible. For this we consider the sequence on $\Spa(L)_{\et}$
\[1\to \mu_p\to 1+\m_L \O^+\xrightarrow{x\mapsto x^p} 1+\m_L \O^+\to 1.\]
Since $H^1_{\et}(\Spa(L),\mu_{p})=1$, this shows that $1+\m_L$ is $p$-divisible.
\end{proof}
An exponential is one of the choices that Faltings makes in \cite{Faltings_SimpsonI} to define his $p$-adic Simpson correspondence on curves. In the process, Faltings shows that the  datum of $\exp$ induces an exponential map for any commutative algebraic group. The goal of this section is to prove a generalisation of this statement for rigid analytic groups, see \Cref{t:exp-for-loc-p-divisible-grp} below.

We now switch back to the algebraically closed field $K$.
From now on, we assume that $G$ is a commutative rigid group over $K$.
\begin{Lemma}\label{l:G[p]-is-fet}
Let $G$ be a quasi-compact commutative rigid group over $K$. Then $G[p^n]$ is a finite \'etale group for each $n\in \mathbb N$. 
\end{Lemma}
\begin{proof}
By \cite[Lemme~1]{Fargues-groupes-analytiques}, any homomorphism $f:G\to H$ of  rigid groups such that the induced map $\Lie(f):\Lie G\to \Lie H$ is bijective is \'etale. It follows that $[p^n]:G\to G$ is \'etale, so $G[p^n]$ is \'etale.  On the other hand, $G[p^n]$ is Zariski-closed in $G$, because it is the pullback of a Zariski-closed subset $1\in G$ along $[p^n]:G\to G$ (see also \cite[Exemple~2]{Fargues-groupes-analytiques}).
It follows that for any affinoid open $U\subseteq G$, the intersection $G[p^n]\cap U$ is affinoid of dimension $0$ in an affinoid, hence finite. As $G$ is quasi-compact, it follows that $G[p^n]$ is finite.
\end{proof}
For any commutative rigid group $G$, we have the topological $p$-torsion subgroup:
\begin{Definition}[{\cite[Proposition~2.14]{heuer-geometric-Simpson-Pic}}]
The topological $p$-torsion subgroup  $\wh{G}\subseteq G$ is the image in v-sheaves of the morphism $\underline{\mathrm{Hom}}(\underline{\Z}_p,G)\xrightarrow{\mathrm{ev}_1} G$.
\end{Definition}
We note that this was written as $G\langle p^\infty \rangle$ in \cite{heuer-geometric-Simpson-Pic}. A priori, $\wh{G}\subseteq G$ is a v-sheaf, but we can also regard it as a rigid group, which we can use  to give a more canonical description of the logarithm map of \Cref{l:exp-isom-on-nbhd-basis}:
\begin{Proposition}[{\cite[\S1.6]{Fargues-groupes-analytiques} and \cite[Proposition~2.14]{heuer-geometric-Simpson-Pic}}]\label{p:Prop2.14}
Let $G$ be a commutative rigid group. 
\begin{enumerate}
\item 
The v-sheaf $\wh{G}$ is representable by an open subgroup $\wh{G}\subseteq G$.
\item\label{l:logG-top-p-torsion} There is a unique homomorphism
\[\log_G: \wh{G}\to \mathfrak g,\]
such that $\Lie(\log_G):\Lie G\to \Lie G$ is the identity. We have $\ker\log_G=G[p^\infty]$.

\end{enumerate}
Assume moreover that there is an admissible formal group $\mathfrak G$ over $\O_K$ with adic generic fibre $G=\mathfrak G_{\eta}^\mathrm{ad}$ as defined in \cite[\S2.2]{ScholzeWeinstein}. Then:
\begin{enumerate}
\setcounter{enumi}{2}
\item The identity component $\widehat{G}^\circ$ is isomorphic to $(\mathfrak G^{\wedge})_{\eta}^\mathrm{ad}$ where $\mathfrak G^{\wedge}$ is the formal completion at the identity.
\item There is a ``connected-\'etale'' short exact sequence of rigid groups
\[1\to \widehat{G}^\circ\to \widehat{G}\to (\mathfrak G[p^\infty]^{\et})^{\mathrm{ad}}_\eta\to 0.\]
\end{enumerate}
\end{Proposition}
\begin{proof}
Parts (1) and (2) are contained in \cite[Proposition~2.14.1 and 2]{heuer-geometric-Simpson-Pic}. We reprove (2) to justify the statement about $\Lie(\log_G)$:
By \Cref{l:exp-isom-on-nbhd-basis} and \cite[Proposition 3.5]{heuer-G-Torsor}, there is an open subgroup $G_0\subseteq G$ for which there exists an injective homomorphism $\log_0:G_0\to \mathfrak g$ with the desired property. For any smooth affinoid algebra $R$ and any $x\in \wh{G}(R)$, one has $x^{p^n}\in G_0(R)$ for some $n\in \mathbb N$. Since $\mathfrak g$  is uniquely divisible, we can define $\log(x)=\frac{1}{p^n}\log_0(x^{p^n})\in \mathfrak g(R)$.

Part (4) can be deduced from the connected-\'etale sequence of the $p$-divisible group $\mathfrak G[p^\infty]$ over $\O_K$ by \cite[Corollaire 13]{Fargues-groupes-analytiques}. 
Alternatively, we can give a \mbox{v-sheaf}-theoretic argument which also shows (3): By the proof of \cite[Proposition 2.14]{heuer-geometric-Simpson-Pic}, $\wh{G}\subseteq G$ is the sub-\mbox{v-sheaf} on $\Perf_K$ given by sheafification of 
\[ (R,R^+)\mapsto \{x \in \mathfrak G(R^+ )\mid \text{$x$ mod }\mathfrak m_K \in \mathfrak  G(R^+/\mathfrak m_K)[p^\infty]\},\]
whereas by \cite[Proposition 2.2.2]{ScholzeWeinstein} and the definition of formal completion, $(\mathfrak G^{\wedge})_{\eta}^\mathrm{ad}$ is the sheafification of
\[ (R,R^+)\mapsto \{x \in \mathfrak G(R^+ )\mid \text{$x\equiv 1$ mod }\mathfrak m_K \text{ in }\mathfrak G(R^+/\mathfrak m_K)\}.\]
Since $(\mathfrak G^{\wedge})_{\eta}^\mathrm{ad}$ is an open disc by \cite[Lemma~3.1.2]{ScholzeWeinstein}, in particular connected, this shows $(\mathfrak G^{\wedge})_{\eta}^\mathrm{ad}\subseteq \widehat{G}^\circ$. On the other hand, since $R^+/\mathfrak m_K$ is reduced, we have \[\mathfrak G(R^+/\mathfrak m_K)[p^\infty]=\mathfrak G[p^\infty]^{\et}(R^+/\mathfrak m_K)=\mathfrak G[p^\infty]^{\et}(R^+)\] by Henselian lifting. This exhibits $(\mathfrak G^{\wedge})_{\eta}^\mathrm{ad}$ as the kernel of a natural map  $\widehat{G}\to (\mathfrak G[p^\infty]^{\et})^{\mathrm{ad}}_{\eta}$ from a rigid group to an \'etale group. Hence $\widehat{G}^\circ\subseteq (\mathfrak G^{\wedge})_{\eta}^\mathrm{ad}$.
\end{proof}

There is also a version of $\wh{G}$ which incorporates torsion coprime to $p$: 
\begin{Definition}[{\cite[Definition~2.5, Proposition~2.14]{heuer-geometric-Simpson-Pic}}]
There is a unique open subgroup $G^{\tt}\subseteq G$ that represents the image of the morphism of v-sheaves
$\underline{\mathrm{Hom}}(\underline{\wh{\Z}},G) \xrightarrow{\mathrm{ev}_1} G$. We call it the topological torsion subgroup of $G$. 
The natural inclusion $\wh G\to G^\tt$ identifies $\wh G$ with an open and closed subgroup of $G^\tt$, and the cokernel is given by the union of all coprime-to-$p$-torsion points of $G$.	
\end{Definition}

\begin{Lemma}\label{l:log-on-Gtt}
There is a left exact sequence of rigid groups, functorial in $G$
\[ 0\to G_{\mathrm{tor}}\to G^\tt\xrightarrow{\log_G} \mathfrak g,\]
such that $\log_G$ induced the identity on tangent spaces. Here $G_{\mathrm{tor}}=\varinjlim_{N\in \N}G[N]$.
\end{Lemma}
\begin{proof}
This follows from \Cref{p:Prop2.14}(\ref{l:logG-top-p-torsion}): We have $G^\tt = \wh{G}\cdot G_{\mathrm{tor}}$ and  define $\log_G=0$ on $G_{\mathrm{tor}}$.
\end{proof}

\subsection{Criteria for $\log_G$ to be surjective}
The goal of this subsection is to prove the following characterisation for when the logarithm morphism is surjective:
\begin{Definition}\label{def:loc-pdiv}
\begin{enumerate}
\item Following Fargues \cite{Fargues-groupes-analytiques}, we say that $G$ is an analytic $p$-divisible group if $G=\widehat{G}$ and $[p]:G\to G$ is finite surjective.
\item More generally, we say that a general rigid group $G$ is locally $p$-divisible if it contains an open subgroup $U$ that is an analytic $p$-divisible group.
\end{enumerate}
\end{Definition}

\begin{Example}\label{ex:alg-G-loc-p-div}
\begin{enumerate}
\item Let $G$ be a commutative connected algebraic group. Then by Rosenlicht's Theorem, it is an extension of an abelian variety by a product of $\G_a$ and tori. It follows that $[p]:G\to G$ is surjective, since this is true for abelian varieties, $\G_a$ and tori. By \Cref{p:log-surj-if-loc-pdiv} below, it follows  that $G$ is locally $p$-divisible.
\item Any extension of locally $p$-divisible groups is locally $p$-divisible. 
\item The closed unit disc with its additive structure,  $G=\G_a^+$, is not locally $p$-divisible: It is topologically $p$-torsion, so $G=\wh{G}$, but clearly $[p]:G\to G$ is not surjective.
\end{enumerate}
\end{Example}
\begin{Proposition}\label{p:log-surj-if-loc-pdiv}
Let $G$ be a commutative rigid group over $K$. Consider the statements:
\begin{enumerate}
\item $[p]:G\to G$ is surjective finite \'etale,
\item $[p]:G\to G$ is surjective,
\item $[p]:\wh{G}\to \wh{G}$ is surjective,
\item $\wh{G}^\circ$ is an analytic p-divisible group,
\item $G$ is locally $p$-divisible, 
\item there exists an open subgroup $U\subseteq G$ on which $[p]:U\to U$ is surjective,
\item $\log_G:\wh{G}\to \mathfrak g$ is surjective.
\end{enumerate}
Then (1) $\Rightarrow$ (2) $\Rightarrow$ (3) $\Rightarrow$ (4) $\Leftrightarrow$ (5) $\Leftrightarrow$ (6) $\Leftrightarrow$ (7). If $\pi_0(G)$ is finite, we have (1) $\Leftrightarrow$ (2). If $G$ is quasi-compact and connected, we have (4) $\Leftrightarrow$ (1), so in this case all statements are equivalent.
\end{Proposition}

\begin{Remark}
We suspect that quasi-compactness is not necessary for $(4)\Rightarrow (1)$ to hold.
\end{Remark}

\begin{proof}[Proof of \Cref{p:log-surj-if-loc-pdiv}]
(1) $\Rightarrow$ (2) is trivial.

(2) $\Rightarrow$ (3): 
For any point $x\in \wh{G}(C,C^+)$ with values in an algebraically closed non-archimedean field extension $(C,C^+)$ of $K$, there exists some $y\in G(C,C^+)$ such that $[p](y)=x$. But then $[p^n](y)=[p^{n-1}](x)\to 0$, so $y\in \wh{G}(C,C^+)$. 

(3) $\Rightarrow$ (6): Trivial as $\wh{G}\subseteq G$ is open.

(4) $\Rightarrow$ (5): We set $U=\widehat{G}^\circ$ which is by assumption analytic $p$-divisible.

(5) $\Rightarrow$ (6) is trivial.

(6) $\Rightarrow$ (7): Since (2) $\Rightarrow$ (3), we may replace $U$ by $\wh{U}$. By \Cref{l:exp-isom-on-nbhd-basis},
the image $M:=\log(U)$ is then an open subgroup of $\mathfrak g$. If $[p]:U\to U$ is surjective, then so is $[p]:M\to M$. But the only open $p$-divisible subgroup of $\mathfrak g$ is the whole group.

(3) $\Rightarrow$ (4): It is clear that $U:=\widehat{G}^\circ$ satisfies $U=\widehat{U}$. By the Snake Lemma (applied in the category of v-sheaves), the cokernel of $[p]:\widehat{G}^\circ\to \widehat{G}^\circ$ is equal to the constant group given by the cokernel of the map $\widehat{G}[p]\to \pi_0(\widehat{G})[p]$. But any morphism from $\widehat{G}^\circ$ to a constant group vanishes. Thus the map $[p]:U\to U$ is surjective. It is \'etale by \cite[Lemme 1]{Fargues-groupes-analytiques}.

It remains to see that this map is moreover finite.  For this, it suffices by \cite[Lemme~5]{Fargues-groupes-analytiques} to prove that $U[p]$ is finite. More precisely, as it is a $p$-torsion abelian group, $U[p]$ is necessarily isomorphic to $(\Z/p\Z)^{\oplus I}$ for some index set $I$, and we need to see that $I$ is finite. We first note that $[p]$ being surjective implies by choosing  successive preimages of basis elements under $[p]$ that $U[p^n]=(\Z/p^n\Z)^{\oplus I}$, hence
\[ U[p^\infty]=(\Q_p/\Z_p)^{\oplus I}.\]
Since we already know that (3) $\Rightarrow$ (7), it follows that we have a short exact sequence of rigid groups
\[ 0\to (\Q_p/\Z_p)^{\oplus I}\to U\to \mathfrak g\to 0.\]
We claim that $U$ being connected forces $|I|<\infty$. To see this, let $m:=\dim G+1$ and consider the pushout $V$ of $U$ along any surjective morphism $(\Q_p/\Z_p)^{\oplus I}\to (\Q_p/\Z_p)^{m}$, so that we have a commutative diagram of rigid groups
\[
\begin{tikzcd}
0 \arrow[r] & (\Q_p/\Z_p)^{\oplus I}\arrow[r] \arrow[d] & U \arrow[d] \arrow[r] & \mathfrak g \arrow[d,equal] \arrow[r] & 0 \\
0 \arrow[r] & (\Q_p/\Z_p)^{m} \arrow[r]           & V \arrow[r]           & \mathfrak g \arrow[r]           & 0.
\end{tikzcd}\]

Then $V$ is still connected, being the image of a  surjective map $U\to V$. Since it is by definition an extension of analytic $p$-divisible groups, $V$ is clearly itself an analytic $p$-divisible group.
By \cite[Théorème 7.2]{Fargues-groupes-analytiques}, this implies that
$m=\mathrm{rk}(V[p^\infty])\leq \dim V=\dim G$, as we can discard the first summand in loc.\ cit.\ since $V$ is connected. This implies that $|I|\leq  \dim G$. Consequently, $U[p]$ is finite, as we wanted to see.

(7) $\Rightarrow$ (4): For this we consider the morphism of short exact sequences
\[
\begin{tikzcd}
0 \arrow[r] & {\wh{G}^\circ [p^\infty]} \arrow[r]                     & \wh{G}^\circ \arrow[r, "\log"]                     & \mathfrak g  \arrow[r]&0                   \\
0 \arrow[r] & {\wh{G}^\circ [p^\infty]} \arrow[r] \arrow[u, "{[p]}"'] & \wh{G}^\circ \arrow[r, "\log"] \arrow[u, "{[p]}"'] & \mathfrak g \arrow[u, "{[p]}"']\arrow[r]&0.
\end{tikzcd}\]
For each $n$, denote by $Q_n$ the cokernel of $[p]:\wh{G}^\circ [p^n]\to \wh{G}^\circ [p^n]$ (sheaf-theoretic, e.g.\ in v-sheaves). Since  $\wh{G}^\circ [p^n]\to \Spa(K)$ is \'etale, it is a constant rigid group, and hence so is $Q_n$. It follows that the cokernel $Q=\varinjlim Q_n$ in the diagram is a constant rigid group. Hence any homomorphism from the connected group $\wh{G}^\circ$ to $Q$ is trivial. This shows that the middle vertical map is surjective. This proves the result because (3) $\Rightarrow$ (4).

(2) $\Rightarrow$ (1): Assume that $\pi_0(G)$ is finite. Then without loss of generality, we may assume that $G$ is connected. The morphism $[p]$ is \'etale by \cite[Lemme 1]{Fargues-groupes-analytiques}. Since (2) $\Rightarrow $ (4), we know that $G[p]$ is finite. Hence (1) follows from \cite[Lemme 5]{Fargues-groupes-analytiques}.

Assume now that $G$ is quasi-compact and connected. Then:

(4) $\Rightarrow$ (2): We use L\"utkebohmert's structure result \cite[Theorem~I]{Lutkebohmert_structure_of_bounded}:  This says that there is an extension $T\to E\to B$ of a quasi-compact group of good reduction $B$ by a rigid torus $T$ such that $G=E/M$ for some lattice $M\subseteq E$. We can use this to reduce to the case of good reduction: It is clear that $[p]:G\to G$ is surjective if and only if $[p]:B\to B$ is surjective. On the other hand, we have 
\[\wh{G} = \wh{E} \times (M\otimes \Q_p/\Z_p).\]
We deduce that $\wh{G}^\circ=\wh{E}^\circ$,
and this is itself an extension
\[ 0\to \wh{T}\to \wh{E}^\circ\to \wh{B}^\circ\to 0.\]
We deduce that $[p]:\wh{B}^\circ\to \wh{B}^\circ$ is surjective if $[p]:\wh{G}^\circ\to \wh{G}^\circ$ is surjective.

We have thus reduced to the case that $G$ has good reduction.  The reduction $\overline{G}$ is then a commutative connected algebraic group, hence an extension of an abelian variety by a product of tori and additive groups. We need to see that the additive part vanishes, then $[p]:\overline{G}\to \overline{G}$ is finite surjective, so the same is true for $[p]:G\to G$ by \cite[\S6.3.5, Theorem~1]{BGR}.

Assume for a contradiction that  $\overline{G}$ contains a subgroup isomorphic to $\G_a$. By \Cref{p:Prop2.14}(4), $\widehat{G}^\circ$ is the analytic $p$-divisible group associated to the connected part of the $p$-divisible group of $G$ over $\O_K$. Consequently, we have a morphism of short exact sequences
\[\begin{tikzcd}%Thanks Yichuan! https://tikzcd.yichuanshen.de/
1 \arrow[r] & \widehat{G}^\circ(K) \arrow[r]           & G(K) \arrow[r]           & \overline{G}(k) \arrow[r]           & 1 \\
1 \arrow[r] & \widehat{G}^\circ(K) \arrow[r] \arrow[u,"{[p]}"] & G(K) \arrow[u,"{[p]}"] \arrow[r] & \overline{G}(k) \arrow[u,"{[p]}"] \arrow[r] & 1.
\end{tikzcd}\]
If $\overline{G}$ contains a factor of $\G_a$, then the kernel of $[p]:\overline{G}(k)\to \overline{G}(k)$ contains a copy of $k$. As the kernel of the middle map is finite, this means by the Snake Lemma that the cokernel of the first map is non-trivial. This contradicts the assumption that $\widehat{G}^\circ$ is $p$-divisible.
\end{proof}
\subsection{Exponentials for $G$ induced by exponentials for $\G_m$}
We can now prove the main result of this section: A choice of an exponential $\exp:K\to 1+\m_K$ for the group $\G_m$ induces compatible choices of exponentials for all locally $p$-divisible rigid groups $G$.

\begin{Theorem}\label{t:exp-for-loc-p-divisible-grp}
Let $K$ be a complete algebraically closed extension of $\Q_p$ and let $G$ be a locally $p$-divisible commutative rigid group over $K$. Then the choice of an exponential $\exp:K\to 1+\m_K$ for $\G_m$ induces a continuous homomorphism
\[\exp_G:\Lie G\rightarrow \widehat{G}(K)\]
splitting $\log_G$ from \Cref{p:Prop2.14}(\ref{l:logG-top-p-torsion}). The association $G\mapsto \exp_G$ is natural in $G$.
\end{Theorem}
This generalises a result of Faltings, who considered the case of commutative algebraic groups \cite[pp 856]{Faltings_SimpsonI}. The latter are locally $p$-divisible groups by \Cref{ex:alg-G-loc-p-div}.
\begin{proof}
By \Cref{p:log-surj-if-loc-pdiv}, the group $\wh{G}^\circ$ is an analytic $p$-divisible group in the sense of Fargues \cite{Fargues-groupes-analytiques}. By \cite[Th\'eor\`eme~3.3]{Fargues-groupes-analytiques}, there is then a natural Cartesian diagram of rigid groups
\[\begin{tikzcd}
\wh{G}^\circ \arrow[d, "\log"'] \arrow[r] & {\HOM(\Lambda(1),\G_m)} \arrow[d, "\log"']                     \\
\mathfrak g \arrow[r]                   & {\HOM(\Lambda(1),\G_a)}, \arrow[u, "\exp"', dotted, bend right]
\end{tikzcd}\]
functorial in $G$, where $\Lambda:=T_p(\wh{G}^\circ)^\vee$.
With this at hand, we can argue as in \cite[Proposition~A.1.2]{xu2022parallel}:
The chosen exponential induces a section of the right map on $K$-points, thus we obtain a section of the left map on $K$-points by Cartesianness.
\end{proof}

\section{The Corlette-Simpson correspondence for commutative $G$}

\subsection{Recollections on the Hodge--Tate short exact sequence for $\G_a$ and $\G_m$}\label{subs:HT}
Assume that $X$ is a proper, smooth rigid space over $K$. Set $\wtOm= R^1\nu_\ast \CO \cong \Omega^1_X(-1)$. We fix a base point $x_0\in X(K)$. In this situation, we have by \cite[\S3]{Scholze2012Survey} the Hodge--Tate sequence
\begin{equation}\label{eq:HT-seq}
0\to H^1_{\et}(X,\G_a)\to H^1_\mathrm{v}(X,\G_a)\to H^0(X,\wtOm)\to 0
\end{equation}
in which the middle term can be naturally identified as
$H^1_\mathrm{v}(X,\G_a)=\Hom(\pi_1^\et(X,x_0),K)$
via a Cartan-Leray spectral sequence for the universal cover $\wt X\to X$ (c.f. \cite[Proposition 2.8.1, Proposition 4.9]{heuer-Line-Bundle}). As we now recall, in general, the sequence doesn't split canonically. However, a splitting can be functorially induced by an additional choice:

\begin{Definition}\label{d:lift} Let $X$ be a smooth proper rigid space over $K$. By a (flat) $B_{\dR}^+/\xi^2$-lift of $X$ we mean a flat morphism of adic spaces $\mathbb X\to \Spa(B_{\dR}^+/\xi^2)$ with an isomorphism \[\mathbb X\times_{\Spa(B_{\dR}^+/\xi^2)} \Spa(K)\isomarrow X.\]
\end{Definition}
We note that such a datum always exists. There is a canonical choice if $X$ admits a model over a discretely valued subfield $K_0\subseteq K$ with perfect residue field.
\begin{Proposition}[{\cite[Proposition 7.2.5]{Guo_HodgeTate}}]\label{p:HT-split}
Any choice of a $B_{\dR}^+/\xi^2$-lift $\mathbb X$ of $X$ induces a splitting of \eqref{eq:HT-seq}. This splitting is functorial in $(X,\mathbb X)$.
\end{Proposition}

There is an analogous story when we replace $\G_a$ by the group $\G_m$: By the main result of \cite{heuer-Line-Bundle}, there is a short exact sequence
\begin{equation}\label{eq:mult-HT-seq}
0\to H^1_{\et}(X,\G_m)\to H^1_\mathrm{v}(X,\G_m)\to H^0(X,\wtOm)\to 0
\end{equation}
classifying the discrepancy between v-line bundles and \'etale line bundles in terms of the additional datum of a Higgs field. This time, the splitting requires not only a $B_\dR^+/\xi^2$-lift $\mathbb X$, but also the datum of an exponential $\exp$ as in \Cref{d:exp}, providing a continuation of the Lie group exponential for $\G_m$.

\subsection{General commutative $G$}

We now prove an analog of the short exact sequence \eqref{eq:mult-HT-seq} for general commutative rigid groups, thus giving a generalisation of the Hodge--Tate sequence \eqref{eq:HT-seq}  and the multiplicative Hodge--Tate sequence \eqref{eq:mult-HT-seq} to a large class of commutative rigid groups:

\begin{Theorem}\label{thm:HodgeTatesplit}
Let $X$ be a connected smooth proper rigid space with a base point $x_0\in X(K)$. Let $G$ be a commutative, locally $p$-divisible rigid group (see \Cref{def:loc-pdiv}) over $K$. 
\begin{enumerate}
\item The Leray spectral sequence for $\nu:X_\mathrm{v}\to X_{\et}$ applied to the sheaf $G$ induces a short exact sequence of abelian groups, natural in $G$,
\begin{equation}\label{eq:CLforG}
	0 \rightarrow H^1_{\et}(X,G) \rightarrow H^1_\mathrm{v}(X,G) \xrightarrow{\HTlog} H^0(X, \wtOm)\otimes_K \Lie G\to 0.
\end{equation}
\item Choices of a $B_{\dR}^+/\xi^2$-lift of $X$ (\Cref{d:lift}) and an exponential for $K$ (\Cref{d:exp}) induce a natural splitting $s_G$ of $\HTlog$. This is functorial in the following sense:  For any $\theta\in H^0(X,\wtOm^1)\otimes_K\Lie G$, these choices induce  a {\rm v}-$G$-torsor $V_{\theta}$ on $X$ with $\HTlog(V_\theta)=\theta$, and for any homomorphism of commutative locally $p$-divisible rigid groups $f:G\to H$ a natural isomorphism
\[V_{\theta}\times^GH\isomarrow V_{(\Lie f)(\theta)}.\]
\end{enumerate}
\end{Theorem}

\begin{proof}
By \cite[Corollary 4.3]{heuer-Moduli}, the Leray spectral sequence for  $\nu:X_\mathrm{v}\to X_{\et}$ applied to $G$ gives a left-exact sequence
\[0 \rightarrow H^1_{\et}(X,G) \rightarrow H^1_\mathrm{v}(X,G) \rightarrow H^0(X, \wtOm)\otimes_K \Lie G.\]

In \Cref{rslt:vbundles}, we have seen that every homomorphism $\rho:\pi_1^{\et}(X,x_0) \rightarrow G(K)$ induces a class in $H^1_\mathrm{v}(X,G)$. Since $\pi_1^{\et}(X,x_0)$ is profinite, it follows from \cite[Proposition~2.14]{heuer-geometric-Simpson-Pic} that the image of $\rho$ is contained in $G^{\tt}(K).$ Composition with the logarithm map 
$\log: G^\tt \rightarrow \mathfrak g$ from \Cref{l:log-on-Gtt} induces a commutative diagram

\begin{center}\begin{tikzcd}
	\Hom(\pi_1^{\et}(X,x_0),G^{\tt}(K)) \arrow[r,"\log"] \arrow[d] & \Hom(\pi_1^{\et}(X,x_0), \Lie G) \arrow[d,"\HT"]\\
	H^1_\mathrm{v}(X, G) \arrow[r,"\HTlog"] & H^0(X, \wtOm) \otimes_K \Lie G		\end{tikzcd}\end{center}
	where the right vertical map is induced by applying $-\otimes \Lie G$ to
	\[\HT:\Hom(\pi_1^{\et}(X,x_0), K)\isomarrow H^1_\mathrm{v}(X,\CO) \rightarrow H^0(X, \wtOm).
	\]
	Here the first identification comes from the Cartan-Leray sequence for $\CO$ along $\wt X\rightarrow X$.
	
	Now the choice of a $B_{\dR}^+/\xi^2$-lift of $X$ induces a splitting of the Hodge--Tate sequence (\Cref{p:HT-split}) and hence a splitting of the right vertical map. By \Cref{t:exp-for-loc-p-divisible-grp}, an exponential for $K$  induces an exponential map $\exp_G:\Lie G \rightarrow \widehat{G}(K)$ that splits the top horizontal map. In particular, both maps are split surjective. Thus $\HTlog$ is also split surjective, with a splitting induced by the above choices.
	
	%To see this, we use that $\log:G^\tt\to \Lie G$ is surjective by \Cref{p:log-surj-if-loc-pdiv}. %Second, since $\Lie G$ is a $K$-vector space, any morphism $\pi_1^\et(X,x)\to \Lie G$ factors through the maximal pro-$p$ quotient $T$ of $\pi_1^\et(X,x)$. By \cite[Theorem~1.1]{Scholze-padic-Hodge} and \cite[Lemma~4.11]{heuer-Line-Bundle}, this is a finite free $\Z_p$-module. This combines to show that any morphism $T\to \Lie G$ lifts to $T\to G^\tt$, hence the top map is split surjective. More precisely, a splitting is induced by any continuous splitting
	%\[ \Lie G\to G^\tt(K).\]
	%By \Cref{t:exp-for-loc-p-divisible-grp}, such a choice is induced by the choice of an exponential on $K$.
	The naturality in $G$ is then immediate from the naturality of $\exp_G$ in \Cref{t:exp-for-loc-p-divisible-grp}, and from the fact that the left vertical map is given via \Cref{rslt:vbundles} by the functorial construction of sending $\rho$ to the pushout $V_{\rho}$ of $\rho$ along the $\pi_1^{\et}(X,x_0)$-torsor $\wt X\to X$.
\end{proof}

We can deduce our final main result, the $p$-adic Simpson correspondence for commutative $G$, giving answers to each of Questions \ref{q3}, \ref{q1}, \ref{q2} for commutative $G$:

\begin{Theorem}\label{rslt:commutativecase} 
Let $X$ be a connected smooth proper rigid variety over $K$ with a base point $x_0\in X(K)$, and let $G$ be a commutative locally $p$-divisible rigid group over $K$. 

\begin{enumerate}
\item Choices of an exponential $\exp$ for $K$ and a $B_{\dR}^+/\xi^2$-lift $\mathbb X$ of $X$  induce an equivalence
\begin{align*}
	\Big\{\text{v-$G$-torsors on $X$} \Big\}\isomarrow  \Big\{
	\text{$G$-Higgs bundles on $X$}\Big\}
\end{align*}
of groupoids that is natural in $G$ and in $(X,\mathbb X)$.
\item In case that every morphism $\wt X\to G$ is constant (for example if $G$ is linear analytic), this restricts to an equivalence of groupoids
\begin{align*}
	\Big\{\begin{array}{@{}c@{}l}\text{continuous homomorphisms  }\\
		\pi_1^{\et}(X,x_0) \rightarrow G(K) \end{array}\Big\}\isomarrow  \Big\{\begin{array}{@{}c@{}l}\text{ pro-finite-\'etale}\\
		\text{$G$-Higgs bundles on $X$}\end{array}\Big\}.
\end{align*}
\end{enumerate}
\end{Theorem}

\begin{proof} 
For simplicity, we write $\pi$ for $\pi_1^{\et}(X,x_0)$ below.

For (1), we observe that it suffices to construct a functor from left to right that induces a bijection on isomorphism classes: Note that since $G$ is commutative, the automorphism group of an object on either side is $G(X)=H^0_\mathrm{v}(X,G)= H^0_{\et}(X,G)$ by \Cref{l:Higgs-for-comm-G}.
By \Cref{thm:HodgeTatesplit}.(2), the choices of an exponential for $K$ and a $B_\dR^+/\xi^2$-lift $\mathbb X$ of $X$ induce a section $s$ of the map $\HTlog$ in \eqref{eq:CLforG}, and hence an isomorphism 	
\[H^1_\mathrm{v}(X,G)\isomarrow H^1_{\et}(X,G)\times H^0(X,\wtOm^1)\otimes_K\Lie G,\]
%	By \eqref{eq:CLforG}, the choices of an exponential for $K$ and a $B_\dR^+/\xi^2$-lift $\mathbb X$ of $X$ induce a section 
%	\[s:H^0(X,\wtOm^1_X)\otimes_K\Lie G\to \Homc(\pi,G(K))\]
%   and in particular an isomorphism
%	\[H^1_\mathrm{v}(X,G)\isomarrow H^1_{\et}(X,G)\times H^0(X,\wtOm^1_X)\otimes_K\Lie G.\]
which moreover can be upgraded to a functorial construction:
For any v-$G$-torsor $V$, consider its class in $H^1_\mathrm{v}(X,G)$. Applying $\HTlog$, we obtain $\theta \in H^0(X,\wtOm^1)\otimes_K\Lie G$. \Cref{thm:HodgeTatesplit}.(2) uses the choice of lift and exponential to give us a v-$G$-bundle $V_{\theta}$ on $X$  with $\HTlog(V_\theta)=\theta=\HTlog(V)$ in a functorial way.
Now let $E$ be the v-$G$ bundle $V\otimes^GV_\theta^{-1}$, c.f. \Cref{l:otime^G}. Since $\HTlog$ is additive with respect to contracted products, we see that $\HTlog(E)=0$. Consequently, $E$ is an \'etale $G$-torsor. The desired functor can now be defined by sending $V$ to $(E,\theta)$.
Here we note that since $G$ is commutative, the Higgs field condition is vacuous by \Cref{l:Higgs-for-comm-G}. 

This construction is indeed functorial: Any morphism $V\to V'$ is an isomorphism, hence $\HTlog(V)=\HTlog(V')$ which implies that $V_\theta$ agrees for $V$ and $V'$. So the statement follows from functoriality of $-\otimes^GV_\theta^{-1}$.

The naturality in $G$ follows from naturality in $G$ in \Cref{thm:HodgeTatesplit}(2).

It remains to check the naturality in (1) with respect to maps $X\to Y$ between smooth proper rigid spaces with a lift $\mathfrak f:\mathbb X \rightarrow \mathbb Y$ over $B_{\dR}^+/\xi^2$. But this follows immediately from functoriality of the splitting of the Hodge--Tate sequence for $\mathbb G_a$ induced by $\mathfrak f$, see \cite[Proposition 7.2.5]{Guo_HodgeTate}. 
%Concretely this means given $f: \mathbb X \rightarrow \mathbb X'$ as above, assume $s$, $s'$ are the sections of the Hodge--Tate sequence induced by $\mathbb X$, $\mathbb X'$. Then there exists an endomorphism $\alpha_f$ of $H^0(X,\wtOm^1_X)$ making the following diagram commutative:
%\[\begin{tikzcd}
%   \Homc(\pi, K)\cong H^1_v(X,\CO) & H^0(X,\wtOm^1_X) \ar[l,"s"']\ar[d,"\alpha_f"]\\
%     & H^0(X,\wtOm^1_X) \ar[ul,shift right, "s'"].
%\end{tikzcd}\]
%This induces a natural transform between the functors 
%\[V\mapsto (V\otimes^G V_{s(\theta)}^{-1}, \theta),\]
%\[V\mapsto (V\otimes^G V_{s'(\theta')}^{-1}, \theta'), \]
%where $\theta'=(\alpha_f\otimes id)(\theta)$. This shows the desired functoriality in $\mathbb X$.

For (2), we use \Cref{rslt:vbundles} which says that continuous homomorphisms $\rho: \pi \rightarrow G^{\tt}(K)$ are equivalent to pro-finite-\'etale $G$-bundles on $X_\mathrm{v}$. Note that by definition, the v-$G$-torsor $V_\theta$ is pro-finite-\'etale for every $\theta \in H^0(X,\wtOm^1)\otimes_K\Lie G$. Hence a v-$G$-torsor $V$ on $X$ with image $\theta$ under $\HTlog$ is pro-finite-\'etale if and only if $V\otimes^{G}V_{\theta}^{-1}$ is pro-finite-\'etale. Applying this to the construction in (1), this implies our claim.  
\end{proof}

\hspace{5ex}

\noindent	
{Ben Heuer \hfill Annette Werner \\
Institut f\"ur Mathematik  \hfill Institut f\"ur Mathematik \\
Goethe-Universit\"at Frankfurt \hfill Goethe-Universit\"at Frankfurt\\
Robert-Mayer-Str. 6-8 \hfill Robert-Mayer-Str. 6-8\\
60325 Frankfurt am Main \hfill 60325 Frankfurt am Main\\	
heuer@math.uni-frankfurt.de \hfill werner@math.uni-frankfurt.de
\\[2ex]
\noindent
Mingjia Zhang \\
Princeton University  \\
Department of Mathematics \\
Fine Hall, Washington Road\\
Princeton NJ 08544-1000 USA\\
mz9413@princeton.edu
\end{document}